\documentclass[runningheads]{llncs}
\usepackage{graphicx}
\usepackage{float}
\usepackage{amsfonts}
\usepackage{amsmath}
\usepackage{amssymb}
\usepackage{mathtools}
\usepackage{commath}
\usepackage{bm}
\usepackage{color}
\usepackage{enumitem}
\usepackage{comment}
\usepackage[T1]{fontenc}
\usepackage[utf8]{inputenc}
\usepackage{multirow}
\usepackage{epsfig}

\newcommand{\R}{\mathbb{R}} 
\newcommand{\bs}{\boldsymbol}

\newcommand{\interp}{R}

\newcommand{\mK}{\mathsf{K}}

\newcommand{\Chi}{\mbox{\Large$\chi$}}

\newcommand\restr[2]{{
		\left.\kern-\nulldelimiterspace 
		#1 
		\right|_{#2} 
}}

\begin{document}
\title{Mapped Variably Scaled Kernels: Applications to Solar Imaging}
%
%
\author{Francesco Marchetti\inst{1}\orcidID{0000-0003-1087-7589} \and
Emma Perracchione \inst{2}\orcidID{0000-0003-2663-7803} \and
Anna Volpara \inst{3}\orcidID{0000-0003-1479-0827} \and
Anna Maria Massone \inst{3}\orcidID{0000-0003-4966-8864} \and
Stefano De Marchi\inst{1}\orcidID{0000-0002-2832-8476}\and
Michele Piana \inst{3}\orcidID{0000-0003-1700-991X}}
\authorrunning{F. Marchetti et al.}
%
\institute{University of Padova, Padua, Italy
\email{\{francesco.marchetti,stefano.demarchi\}@unipd.it}
 \and
Polytechnic of Torino, Turin, Italy
\email{emma.perracchione@polito.it}
\and
University of Genova, Genoa, Italy
\email{\{volpara,massone,piana\}@dima.unige.it}
}
\maketitle              
\begin{abstract}
Variably scaled kernels and mapped bases constructed via the so-called fake nodes approach are two different strategies to provide adaptive bases for function interpolation. In this paper, we focus on kernel-based interpolation and we present what we call mapped variably scaled kernels, which take advantage of both strategies. We present some theoretical analysis and then we show their efficacy via numerical experiments. Moreover, we test such a new basis for image reconstruction tasks in the framework of hard X-ray astronomical imaging.

\end{abstract}

\section{Introduction}

Kernel-based interpolation is an effective approach to deal with the scattered data interpolation problem, where data sites do not necessarily belong to some particular structure or grid \cite{Fasshauer,Wendland05}. Therefore, because of its flexibility and the achievable accuracy, it finds application in many different contexts \cite{Guastavino20}, including image reconstruction and the numerical solution of partial differential equations. The effectiveness of radial kernel-based interpolation, which is also known as Radial Basis Function (RBF) interpolation, very often relies on a good choice of the so-called \textit{shape parameter}, which rules the shape of basis functions. However, in many situations, a fine tuning of this shape parameter is not sufficient to construct basis functions that are tailored with respect to both the data distribution and the target function to be recovered. 

Therefore, in order to gain more adaptivity in the interpolation process, Variably Scaled Kernels (VSKs) have been introduced in \cite{Bozzini15}, and further analyzed in \cite{campivsk} in the more general framework of kernel-based regression networks, and in \cite{DeMarchi22} in the context of persistent homology. The VSK setting has also been employed in the reconstruction of functions presenting jumps, leading to the definition of Variably Scaled Discontinuous Kernels (VSDKs) \cite{vskjump}, which turned out to be effective in medical image reconstruction tasks in the field of Magnetic Particle Imaging (MPI) \cite{vskmpi}. 

An alternative approach for constructing data-dependent or target-dependent basis functions, the so-called Fake Nodes Approach (FNA), has been introduced in \cite{DeMarchi20} in the framework of univariate polynomial interpolation, and then extended to rational barycentric approximation \cite{Berrut20} and general multivariate interpolation, including the RBF framework \cite{DeMarchi21}. In particular, in latter paper the authors showed that the VSDK setting and the FNA can lead to very similar results when facing the Gibbs phenomenon in the reconstruction of discontinuous functions. Moreover, in \cite{DeMarchi21b} an effective scheme for dealing with both Gibbs and Runge's phenomena has been proposed.

In this paper, our purpose is to design a unified kernel-based approach for function interpolation by taking advantage of both the VSK setting and the FNA, which are recalled in Section \ref{sec:vsk} and \ref{sec:fake}, respectively. The resulting kernels that we call Mapped VSKs (MVSKs) are defined in Section \ref{sec:mvsdk} and provided with an original theoretical contribution, which includes a focus on the discontinuous case tested in Section \ref{sec:experiments}. In Section \ref{sec:stix} we present the application of the proposed method to the inverse problems of solar hard X-ray imaging and, specifically, we consider data from the on the ESA {\em{Spectrometer/Telescope for Imaging X-rays (STIX)}} telescope, on board of Solar Orbiter mission. Finally, in Section \ref{sec:conclusions} we draw some conclusions.

\section{Variably Scaled Kernel-Based Approximation}\label{sec:vsk}

We refer to \cite{Fasshauer,Wendland05} for the following introduction. 

Let $ \Omega \subseteq \mathbb{R}^{d}$, $d\in\mathbb{N}$, and $ X=X_N = \{  \boldsymbol{x}_i, \; i = 1,  \ldots , N\} \subset \Omega$ be a set of possible scattered distinct nodes, $N\in\mathbb{N}$. Suppose that we wish to reconstruct an unknown function  $f: \Omega \longrightarrow \mathbb{R}$ from its values at $X_N$, i.e. from the vector $\bs{f}=(f(\bs{x}_1),\dots,f(\bs{x}_n))^{\intercal}=(f_1,\dots,f_n)^{\intercal}$. In kernel-based interpolation, this is obtained by considering an approximating function of the form
\begin{equation*}
    \interp_{f,X}(\bs{x})=\sum_{i=1}^n c_i \kappa_{\varepsilon}(\bs{x},\bs{x}_i),\quad \bs{x}\in\Omega,
\end{equation*}
where $\bs{c}=(c_1,\dots,c_n)^{\intercal}\in\mathbb{R}^N$ and $\kappa_{\varepsilon}:\Omega\times\Omega\longrightarrow \mathbb{R}$ is a strictly positive definite kernel, which depends on a \textit{shape parameter} $\varepsilon>0$. In the following, we may use the shortened notation $\kappa=\kappa_{\varepsilon}$. The interpolation conditions are imposed by employing $\bs{c}$ so that
\begin{equation}\label{eq:system}
    \mK\bs{c}=\bs{f},
\end{equation}
where $\mK=(\mK_{i,j})=\kappa(\bs{x}_i,\bs{x}_j)$, $i,j=1,\dots,N$, is the so-called interpolation (or collocation or simply kernel) matrix. The vector $\bs{c}$ that satisfies \eqref{eq:system} is unique as long as $\kappa$ is strictly positive definite. Moreover, we assume $\kappa_{\varepsilon}$ to be \textit{radial}, i.e. there exists a univariate function $\varphi_{\varepsilon}:\mathbb{R}_{\ge0}\longrightarrow \mathbb{R}$ such that $\kappa_{\varepsilon}(\bs{x},\bs{y})=\varphi_{\varepsilon}(r)$, with $r\coloneqq\lVert \bs{x}-\bs{y}\lVert_2$.

The interpolant $\interp_{f,\mathcal{X}}$ belongs to the dot-product space $$\mathcal{H}_{\kappa} =  \mathrm{span} \left\{ \kappa(\cdot,\boldsymbol{x}), \; \boldsymbol{x} \in \Omega \right\},$$ whose completion with respect to the norm $\lVert \cdot \lVert_{\mathcal{H}_{\kappa}}=\sqrt{(\cdot,\cdot)_{\mathcal{H}_{\kappa}}}$ induced by the bilinear form $(\cdot,\cdot)_{\mathcal{H}_{\kappa}}$ is the \textit{native space} $\mathcal{N}_{\kappa}$ associated to $\kappa$.  The well-known pointwise error bound \cite[Theorem 14.2, p. 117]{Fasshauer}
\begin{equation}\label{eq:the_error}
    |f(\bs{x})-\interp_{f,X}(\bs{x})|\le P_{\kappa,X}(\bs{x})\lVert f \lVert_{\mathcal{N}_{\kappa}},\quad f\in \mathcal{N}_{\kappa}, \quad \bs{x}\in\Omega,
\end{equation}
involves the \textit{power function} $P_{\kappa,X}=\lVert \kappa(\cdot,x)- \bs{\kappa}(\cdot)^{\intercal}\mK^{-1} \bs{\kappa}(\bs{x})\lVert_{\mathcal{H}_{\kappa}}$, where $\bs{\kappa}(\bs{x})\coloneqq(\kappa(\bs{x},\bs{x}_1),\dots,\kappa(\bs{x},\bs{x}_N))^{\intercal}$. In the estimate \eqref{eq:the_error}, a noteworthy property is that the error is split into a first term that only depends on the nodes and on the kernel, and a second term that relies on the underlying function $f$. 

A different perspective is provided by the following error bound, which takes into account the so-called \textit{fill distance} 
\begin{equation*}
    h_{X,\Omega} \coloneqq  \sup_{ \boldsymbol{x} \in \Omega} {\min_{ \boldsymbol{x}_k  \in {X}} \lVert \boldsymbol{x} - \boldsymbol{x}_k \lVert_2 }.
\end{equation*}
Assuming $\Omega$ to be bounded and satisfying an interior cone condition, and $\kappa \in C^{2k}\left(\Omega \times \Omega\right)$, there exists $C_\kappa(\bs{x})$ such that
	\begin{equation}\label{eq:conv_rate}
		|f\left(\boldsymbol{x}\right)-\interp_{f,X}\left(\boldsymbol{x}\right)| \leq C_\kappa(\bs{x}) h^k_{ X,\Omega}  \lVert f \lVert_{{\cal N}_{\kappa}(\Omega)}.
	\end{equation}
The factor $C_\kappa(\bs{x})$ depends on the maximum of kernel derivatives of degree $2k$ in a neighborhood of $\bs{x}\in\Omega$, and $h^k_{ X,\Omega}$ needs to be \textit{small enough}; we refer to \cite[Section 14.5]{Fasshauer} for a detailed presentation of this result.

While the theoretically achievable convergence rate is influenced by the fill distance and the smoothness of the kernel, two terms play an important role in affecting the conditioning of the interpolation process:
 \begin{itemize}
     \item 
     The \textit{separation distance} $q_{X} \coloneqq \dfrac{1}{2} \min_{  i \neq j} \lVert \boldsymbol{x}_i - \boldsymbol{x}_j \lVert_2.$
     \item
     The value of the shape parameter $\varepsilon$ of the kernel $\kappa_\varepsilon$.
 \end{itemize}
Precisely, the interpolation process gets more ill-conditioned as the separation distance becomes smaller, which is what usually happens in practice when increasing the number of interpolation nodes, thus reducing the value of the fill-distance (this is often denoted as a \textit{trade-off} principle in RBF literature). As far as the shape parameter is concerned, a large value produces very localized basis functions that lead to a well-conditioned but likely inaccurate approximation scheme, while by lowering such value we may obtain a more accurate reconstruction at the price of an ill-conditioned setting. As a consequence, the tuning of the shape parameter is a non trivial problem in kernel-based interpolation.

In order to partially overcome such instability issues, in \cite{Bozzini15} the authors introduced Variably Scaled Kernels (VSKs), which are defined as follows. Letting  
$\psi: \Omega \longrightarrow \R$ be a scaling or shape function and $\kappa_\varepsilon: (\Omega\times\R)  \times (\Omega\times\R) \longrightarrow \mathbb{R}$ a kernel, a VSK $\kappa_{\psi}: \Omega  \times \Omega \longrightarrow \mathbb{R}$ is defined as
\begin{equation*}
	\kappa_{\psi}(\bs{{x}},\bs{{y}})\coloneqq \kappa_\varepsilon((\bs{{x}},\psi(\bs{x})),(\bs{y},\psi(\bs{y})),
\end{equation*}
for $\bs{x},\bs{y}\in\Omega$. Note that we can consider the related native space $\mathcal{N}_{\kappa_{\psi}}$ spanned by the functions $\kappa_{\psi}(\cdot,\bs{{x}})$, $\bs{x}\in\Omega$. Among the properties of VSKs, we recall that:
\begin{enumerate}
    \item[A1.] The theoretical analysis of the VSK setting reduces to the analysis of the classical framework in the augmented space $\Omega\times\R$.
    \item[A2.] The spaces $\mathcal{N}_{\kappa}$ and $\mathcal{N}_{\kappa_{\psi}}$ are isometrically isomorphic.
\end{enumerate}

The shape parameter is often set to $\varepsilon=1$ in the VSK framework. Therefore, the role played by the shape function is duplex. On the one hand, the tuning of the shape parameter is substituted with the choice of the shape function. On the other hand, $\psi$ can lead to an improvement in the conditioning of the approximation process by increasing the value of the separation distance in the augmented domain $\Omega\times\R$. In the following, we do not necessarily restrict to a fixed value of the shape parameter, as also done in the context of moving least squares \cite{Esfahani23}. Note that  $\kappa_{\psi}$ is (strictly) positive definite if so is $\kappa$.

\section{Interpolation via Mapped Bases}\label{sec:fake}

In this section, we present the so called Fake Nodes Approach (FNA) for the kernel-based approximation framework; for a more general and comprehensive treatment we refer to \cite{DeMarchi21}.

Let $S: \Omega \longrightarrow \mathbb{R}^d$ be an injective map. Our purpose is to construct an interpolant $\interp^S_{f,X}$ of the function $f$ in the space
$$\mathcal{H}_{\kappa^S} =  \mathrm{span} \left\{ \kappa^S(\cdot,\boldsymbol{x}), \; \boldsymbol{x} \in \Omega \right\},$$
where $\kappa^S(\boldsymbol{x},\bs{y})\coloneqq \kappa(S(\bs{x}),S(\bs{y}))$. We have
\begin{equation*}
\interp^S_{f,X}(\boldsymbol{x})=  \sum_{i=1}^N c^S_i \kappa^S(\boldsymbol{x},\boldsymbol{x}_i) = \sum_{i=1}^N c^S_i \kappa(S(\boldsymbol{x}),S(\boldsymbol{x}_i))=\interp_{g,X}(S(\boldsymbol{x})), \quad \boldsymbol{x} \in \Omega,
\end{equation*}
where $g_{|S(X)}=f_{|X}$. In other words, the construction of the interpolant $\interp^S_{f,X}  \in \mathcal{H}_{\kappa^S}$ is equivalent to the construction of a classical interpolant $\interp_{g,S(X)} \in \mathcal{H}_{\kappa}$ at the \emph{fake} nodes $S(X)$. Similarly to the VSK setting, the FNA is provided with the following properties (cf. A1 and A2):
\begin{enumerate}
    \item[B1.] The theoretical analysis of the FNA interpolant reduces to the analysis of the interpolant in the classical framework. This was proved in \cite[Proposition 3.4]{DeMarchi21} and it is linked to the inheritance property of the \textit{Lebesgue constant} \cite{BLSE:RN017988920}.
    \item[B2.] The spaces $\mathcal{N}_{\kappa}$ and $\mathcal{N}_{\kappa^{S}}$ are isometrically isomorphic (a direct consequence of Theorem \ref{thm:iso} and Proposition \ref{prop:cases} stated below).
\end{enumerate}
In previous works, the FNA has been mainly employed for two main purposes. The first is using $S$ to obtain an interpolation design $S(X)$ that leads to a more stable interpolation process with respect to the original set of nodes. For example, this can be achieved in a polynomial-based framework by mapping onto the set of (tensor-product) Chebyshev-Lobatto nodes in the (multi) one-dimensional case. In addition, as we will discuss in Subsection \ref{sec:mvsdk}, $S$ can be constructed in a target-dependent fashion to emulate the possible discontinuities of the underlying function and thus recover accuracy near jump points.

\section{Mapped VSKs}\label{sec:mvsdk}

\subsection{The general framework}

In Sections \ref{sec:vsk} and \ref{sec:fake}, we outlined two different approaches, which however present some similarities and are employed for analogous purposes. In the following, we discuss how the VSK setting and the FNA can be merged in a unified framework. We start by giving the following definition.

\begin{definition}\label{def:mvsks}
    Let $S: \Omega \longrightarrow \mathbb{R}^d$ be an injective map, $\psi: \Omega \longrightarrow \R$ be a shape function and let $\kappa_\varepsilon: (S(\Omega)\times\R)  \times (S(\Omega)\times\R) \longrightarrow \mathbb{R}$ be a kernel. Then, a Mapped VSK (MVSK) $\kappa^S_{\psi}: \Omega  \times \Omega \longrightarrow \mathbb{R}$ is defined as
\begin{equation*}
	\kappa^S_{\psi}(\bs{{x}},\bs{{y}})\coloneqq \kappa_\varepsilon((S(\bs{{x}}),\psi(\bs{x})),(S(\bs{y}),\psi(\bs{y})),
\end{equation*}
for $\bs{x},\bs{y}\in\Omega$.
\end{definition}
We remark that $\kappa^S_{\psi}$ might be defined in different possible equivalent manners. However, the advantage of Definition \ref{def:mvsks} lies in the separation between the actions of $S$ and $\psi$: The function $S$ works in the original dimension $\mathbb{R}^d$, while $\psi$ rules the coordinate of the input in the augmented dimension. Under certain assumptions, a MVSK reduces to a mapped or VSK.
\begin{proposition}\label{prop:cases}
    Let $\kappa^S_{\psi}$ be a MVSK on $\Omega \times \Omega$ built upon a radial kernel $\kappa_\varepsilon$. Then:
    \begin{enumerate}
        \item 
        If $\psi(\bs{x})\equiv \alpha\in\mathbb{R}$, then $\kappa^S_{\psi}=\kappa^S$.
        \item 
        If $S(\bs{x})-S(\bs{y})= \bs{x}-\bs{y}$, (e.g., $S$ is the identity map), then $\kappa^S_{\psi}=\kappa_{\psi}$.        
    \end{enumerate}
\end{proposition}
\begin{proof}
    By hypothesis, there exists $\varphi:[0,+\infty)\longrightarrow\mathbb{R}$ such that $\kappa_\varepsilon(\bs{x},\bs{y})=\varphi(\lVert \bs{x}-\bs{y}\lVert_2)$. Therefore, we can write
    \begin{equation*}
    \begin{split}
        \kappa^S_{\psi}(\bs{{x}},\bs{{y}}) &= \varphi(\lVert (S(\bs{{x}}),\psi(\bs{x}))-(S(\bs{{y}}),\psi(\bs{y}))\lVert_2)\\ &=
    \varphi(\sqrt{(S(\bs{x})-S(\bs{y}))^2+(\psi(\bs{x})-\psi(\bs{y}))^2}),
    \end{split}
    \end{equation*}
    from which the two theses follow.
\end{proof}
We also prove the following.
\begin{theorem}\label{thm:iso}
    The spaces $\mathcal{H}_{\kappa}$ and $\mathcal{H}_{\kappa^S_{\psi}}$ are isometric.
\end{theorem}
\begin{proof}
    By defining the map $\Lambda_\psi^S(\bs{x})\coloneqq (S(\bs{x}),\psi(\bs{x}))$ we can see $\kappa^S_{\psi}$ as the push-forward of $\kappa$ in the sense provided in \cite[Equation 2.52]{Saitoh16}. Then, since $S$ is injective, the proof follows from \cite[Theorem 2.9]{Saitoh16}.
\end{proof}
As a consequence of Theorem \ref{thm:iso}, the native spaces $\mathcal{N}_{\kappa}$ and $\mathcal{N}_{\kappa^S_{\psi}}$ are isometrically isomorphic (cf. \cite[Section 3]{Bozzini15}).

Independently of the target function to be recovered, we recall that the conditioning of the interpolation problem is related to the $\ell_2$-conditioning of the kernel matrix $\mathrm{cond}(\mK)=\lambda_{\max}/\lambda_{\min}$, being $\lambda_{\max},\;\lambda_{\min}$ the maximum and minimum eigenvalue, respectively. It is known that $\lambda_{\min}$ decays according to the separation distance $q_X$, in a way that is influenced by the regularity of the kernel (see \cite[Chapter 16]{Fasshauer}). In this direction, MVSKs can be employed in order to increase the separation distance and thus improve the conditioning of the interpolation scheme, e.g., by separating clustered nodes. On the other hand, diminishing the fill distance may improve the accuracy of the method (see \eqref{eq:conv_rate}). We will experiment on this in Section \ref{sec:experiments}.

\subsection{Working with mapped discontinuous kernels}

In the following, we focus on the case of discontinuous functions. First, we briefly review in which manners VSKs and the FNA were used in the discontinuous setting.

\subsubsection{Variably Scaled Discontinuous Kernels (VSDKs).} The idea proposed in \cite{vskjump} and further investigated in \cite{vskmpi} was to define the scaling function $\psi$ to be discontinuous at the jumps of the target function $f:\Omega\longrightarrow\R$. In order to do so, we assume $\Omega\subset \R^d$ to be a bounded set such that:
\begin{itemize}
    \item $\Omega$ is the union of $m$ pairwise disjoint sets $\Omega_k$, $k \in \{1, \ldots, m\}$.
    \item Each subset $\Omega_k$, $k=1,\dots,m$, has a Lipschitz boundary.
    \item The discontinuity points of $f$ are contained in the union of the boundaries of the subsets $\Omega_k$, $k=1,\dots,m$.
\end{itemize}
Then, letting $\bs{\alpha} = (\alpha_1, \ldots, \alpha_m)$, $\alpha_i \in \R$, the scaling function $\psi$ is such that:
\begin{itemize}
    \item $\psi$ is piecewise constant, such that $\restr{\psi(\bs{x})}{\Omega_k} = \alpha_k$. 
    \item $\alpha_i \neq \alpha_j$ if $\Omega_i$ and $\Omega_j$ are neighboring sets.
\end{itemize}
The theoretical analysis carried out in the referring papers then focused on radial kernels whose related univariate function $\varphi$ has the following Fourier decay
\begin{equation} \label{eq:fourier}
(\mathrm{F} {\varphi})(\bs{\omega}) \sim (1+\lVert\bs{\omega}\lVert^2)^{-s-\frac{1}{2}}, \quad s > \frac{d-1}{2}.
\end{equation}
The native space of the kernels that satisfy \eqref{eq:fourier}, e.g. Matérn and Wendland kernels, is a Sobolev space \cite[Chapter 10]{Wendland05}. In order to present an error bound in terms of Sobolev spaces norm for the VSDK setting, we introduce two necessary ingredients:
\begin{enumerate}
    \item 
    The \textit{regional} fill-distance
    \[ h_k \coloneqq h_{X,\Omega_k} = \sup_{\bs{x} \in \Omega_k} \min_{\bs{x}_k \in {X} \cap \Omega_k} \lVert \bs{x} - \bs{x}_k \lVert_2\]
    and the \textit{global} fill distance
    \[ h \coloneqq \max_{k \in \{1, \ldots, m\}} h_k.\]
    \item
    Letting $s \geq 0$ and $1 \leq p \leq \infty$, we define the space
    \begin{equation*}
    \mathcal{W}\mathcal{P}_p^s(\Omega) \coloneqq \left\{ f: \Omega \longrightarrow \R \:| \: \restr{f}{\Omega_k} \in \mathcal{W}_p^s(\Omega_k),\; k \in \{1, \ldots, m\} \right\},
    \end{equation*}
    which contains the piecewise smooth functions $f$ on $\Omega$ whose restriction to any subregion $\Omega_k$, $i=k,\dots,m$, is contained in the standard Sobolev space $\mathcal{W}_p^s(\Omega_k)$. The space $\mathcal{W}\mathcal{P}_p^s(\Omega)$ is endowed with the norm
    \begin{equation*}
    	\lVert f \lVert_{\mathcal{W}\mathcal{P}_p^s(\Omega)}^p = \sum_{k = 1}^m \lVert \restr{f}{\Omega_k}\lVert_{\mathcal{W}_p^s(\Omega_k)}^p.
    \end{equation*}
\end{enumerate}

In the outlined assumptions and letting $\interp^\psi_{f,X}$ be the kernel-based interpolant of $f$ at $X$ built upon the VSDK $\kappa_\psi$, in \cite[Theorem 3.4]{vskmpi} the authors proved the following. Let $s > 0$, $1 \leq q \leq \infty$ and $t\in\mathbb{N}_0$ such that $\lfloor s \rfloor > t + \frac{d}{2}$. Then, for $f \in \mathcal{W}\mathcal{P}_2^s(\Omega)$ and \textit{sufficiently small} $h$ we have that
\begin{equation} \label{eq:errorestimate}
\lVert f - \interp^\psi_{f,X}\lVert_{\mathcal{W}\mathcal{P}_q^{t}(\Omega)} \leq C h^{s-t-d(1/2-1/q)_+} \lVert f\lVert_{\mathcal{W}\mathcal{P}_2^{s}(\Omega)},
\end{equation}
where the constant $C > 0$ is independent of $h$.

\subsubsection{The S-Gibbs map in the FNA.} In the mapped bases approach, similarly to the VSDK framework, the intuition is to map the nodes in order to create some \textit{gaps} in presence of discontinuities. To do so, considering the collection of subsets $\Omega_1,\dots,\Omega_m$ employed to construct VSDKs, the so-called \textit{S-Gibbs map} is designed as
\begin{equation*}
	S(\bs{x}) = \bs{x} +  \sum_{k=1}^m \bs{\beta}_k \Chi_{\Omega_k}(\bs{x}), 
\end{equation*}
where $\bs{\beta}_k = (k\beta,\dots,k\beta)\in\mathbb{R}^d$, $\beta\in\mathbb{R}$, and $\Chi_{\Omega_k}$ is the characteristic function corresponding to $\Omega_k$. In \cite[Section 4.2]{DeMarchi21}, the authors discussed the analogies between a VSDK $\kappa_\psi$ and a mapped kernels $\kappa^S$ constructed via the S-Gibbs map. It turned out that these two approaches can lead to \textit{similar} results for certain values of the vectors of parameters $\bs{\alpha}$ and $\bs{\beta}_k$. This is due to the fact that the interpolation matrices are \textit{close} being the kernel radial.

\subsubsection{Mapped VSDKs (MVSDKs).} In the mixed approach with the mapped VSK kernel $\kappa_\psi^S$, we deal with the jumps of the underlying function as follows.
\begin{itemize}
    \item We define $\psi$ as in the VSDKs framework. Therefore, the role of the shape function is to mimic the jumps of the target function and thus to prevent the appearance of the Gibbs phenomenon.
    \item Since the discontinuities are already addressed by $\psi$, we employ $S$ to map the set of nodes $X$ to obtain improved distributions locally on each $\Omega_k$, meaning that we aim at diminishing the global fill distance and increasing the global separation distance defined as
    \[ q \coloneqq \min_{k \in \{1, \ldots, m\}} q_k,\]
    being $q_k\coloneqq q_{X_k}= \dfrac{1}{2} \min_{  i \neq j} \lVert \boldsymbol{x}_i - \boldsymbol{x}_j \lVert_2,\; \bs{x}_i,\bs{x}_j\in\Omega_k$.
\end{itemize}
Recalling the error estimate in \eqref{eq:errorestimate}, this proposed construction for MVSDKs may lead to better results, being $h$ smaller. Moreover, by increasing the separation distance, an improvement in the conditioning of the scheme with respect to classical VSDKs is likely to be obtained. We test these aspects in some numerical examples in the next section.

\section{Numerical tests with MVSDKs}\label{sec:experiments}

The purpose of this section is to provide a numerical example to show the benefits of the MVSDK framework in comparison to VSDKs and classical RBF interpolation. To do so, we set $\Omega=[-1,1]^2$ and letting $\bs{x}=(x_1,x_2)$ we consider the target function
\begin{equation*}
	f:\Omega\longrightarrow\R,\quad f(\bs{x})=\left \{
	\begin{array}{ll}
	x_1+x_2, & \quad x_1<-0.3,\\
	\sin(x_1-2x_2),  & \quad 0\le x_1<0.5,\\
	0, & \quad  \textrm{otherwise.}\\
	\end{array} \right.
\end{equation*}
To deal with the jumps of $f$, we define the shape function
\begin{equation*}
	\psi:\Omega\longrightarrow\R,\quad \psi(\bs{x})=\left \{
	\begin{array}{ll}
	0, & \quad x_1<-0.3,\\
    1, & \quad -0.3\le x_1 < 0,\\
	2,  & \quad 0\le x_1<0.5,\\
	3, & \quad  x_1\ge 0.5.\\
	\end{array} \right.
\end{equation*}
As far as the nodes are concerned, we let $G_N$ be a set of $N$ nodes that are sampled from a bivariate normal distribution with mean $\bs{\mu}=(0,0)$ and covariance matrix $\mathsf{\Sigma}=0.1\cdot\mathsf{I}$, being $\mathsf{I}$ the $2\times 2$ identity matrix. We can then consider the following mapping function $S:\Omega\longrightarrow \Omega$ defined as
\begin{equation*}
    S(x_1,x_2)=\bigg(1+\mathrm{erf}\bigg(\frac{x_1}{\sqrt{0.2}}\bigg),1+\mathrm{erf}\bigg(\frac{x_2}{\sqrt{0.2}}\bigg)\bigg)-1,
\end{equation*}
where erf is the well-known \textit{error function}. To clarify the idea behind the construction of $S$, it is known from classical probability theory that if $z_1,\dots,z_n$ are sampled according to a normal distribution of mean $\mu$ and standard deviation $\sigma$, then  $0.5(1+\textrm{erf}((z_i-\mu)/(\sqrt{2}\sigma)))$, $i=1,\dots,n$, are distributed uniformly in $[0,1]$. Therefore, our set of nodes $G_N$ is mapped to a uniform distribution in the square $[-1,1]^2$, as displayed in Figure \ref{fig:nodi_distribuiti}. We point out that possible nodes in $G_N$ that are not in $\Omega$ are removed from the set. 
\begin{figure}
    \centering
    \includegraphics[scale=0.40]{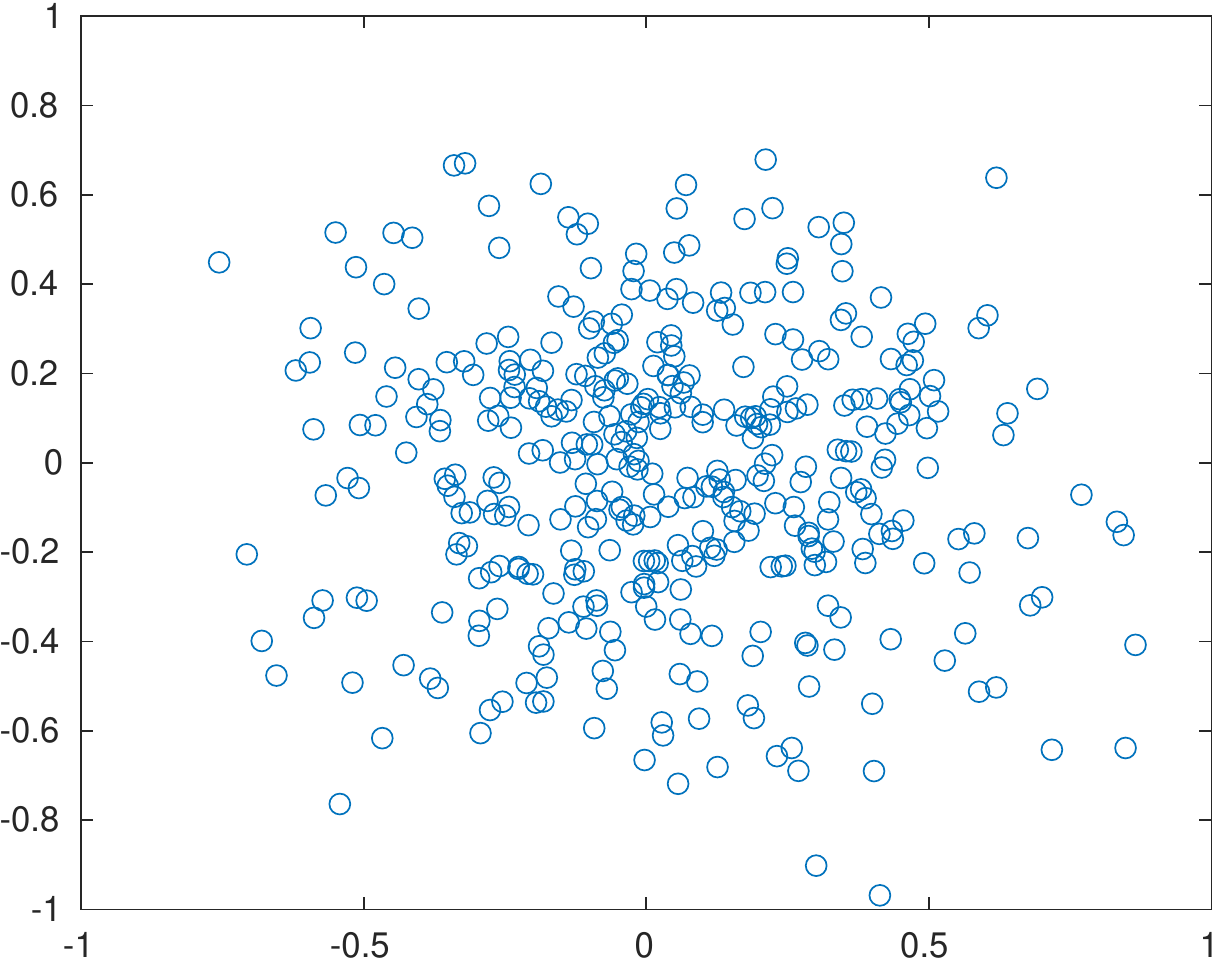}
    \includegraphics[scale=0.40]{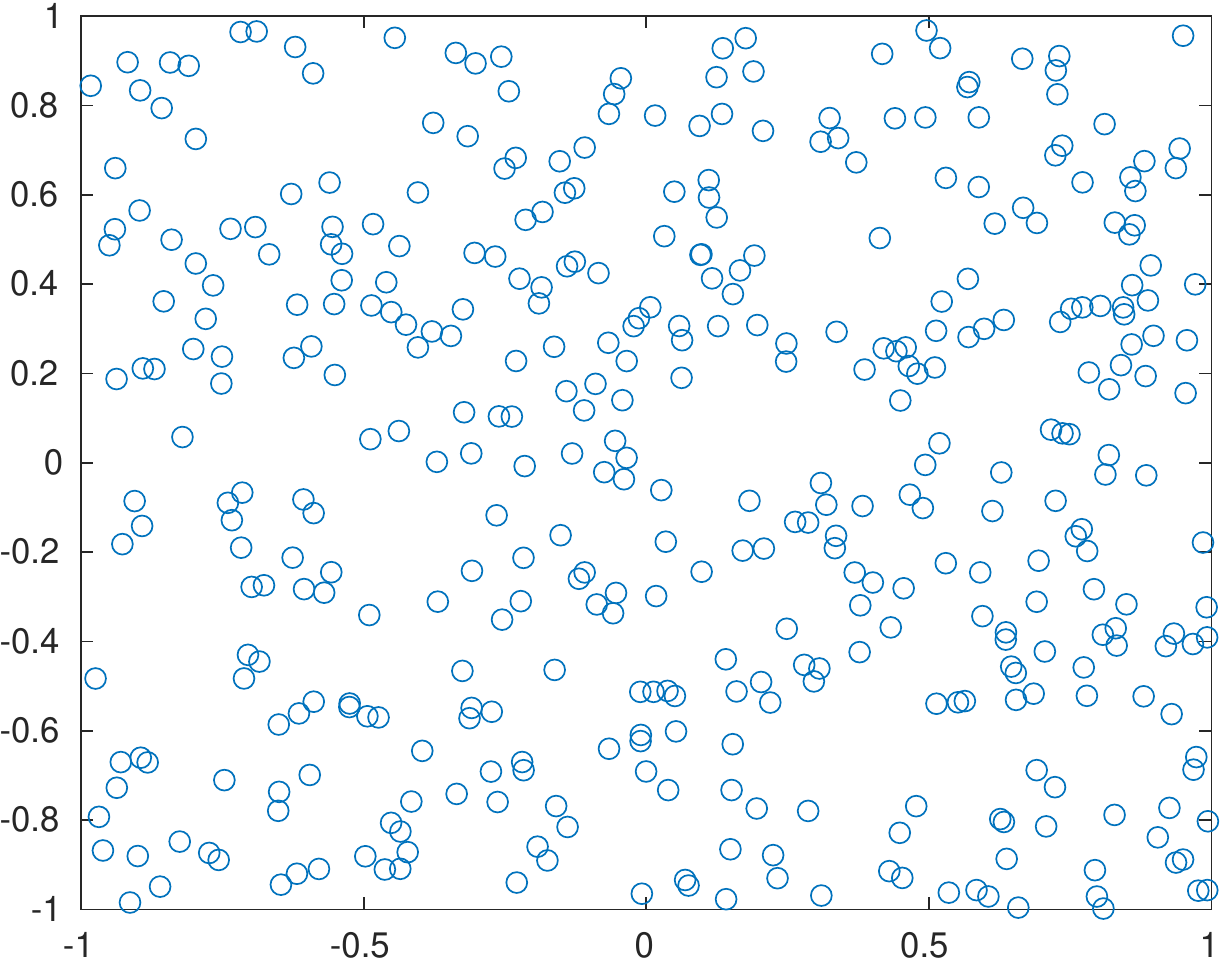}      
    \caption{$N=400$. Left: $G_N$. Right: $S(G_N)$.}
    \label{fig:nodi_distribuiti}
\end{figure}
Consequently, the mapped set $S(G_N)$ is very likely to present smaller fill distance and larger separation distance than $G_N$. The interpolation results are evaluated on a finer $M\times M$ equispaced grid $\Xi_{M^2}$ in $\Omega$. Precisely, we compute the Root Mean Square Error (RMSE)
\begin{align}\label{mae_rmse}
	{\rm RMSE} = \sqrt{\frac{1}{M^2}\sum_{k=1}^{M^2} \left(f(\bs{\xi}_k)-\iota(\bs{\xi}_k)\right)^2},
\end{align}
where $\bs{\xi}_k\in\Xi_{M^2}$ and $\iota=\interp_{f,G_N},\;\interp^\psi_{f,G_N},\; \interp^{\psi,S}_{f,G_N}$ are the interpolants constructed via the classical, the VSDK and the MVSDK. The shape parameter $\varepsilon$ is chosen between $200$ equispaced values in the interval $[0.01,50]$ via Leave-One-Out Cross Validation (LOOCV) \cite{Ling22}. Finally, we let $N$ vary between $10$ and $500$ and we test two radial kernels:
\begin{equation*}
\begin{split}
    &  \varphi_{W}(r)=(1-\varepsilon r)_+^2\quad\textrm{Wendland $C^0$,}\\
    &  \varphi_{M}(r)=e^{-\varepsilon r}(15+15\varepsilon r+6(\varepsilon r)^2+(\varepsilon r)^3)\quad\textrm{Matérn $C^6$.}  
\end{split}
\end{equation*}

In Figure \ref{fig:sep_fill}, we show the behavior of the separation and fill distances, while in Figure \ref{fig:rmse} we show the RMSEs achieved with both $\varphi_{W}$ and $\varphi_{M}$.

\begin{figure}
    \centering
    \includegraphics[scale=0.40]{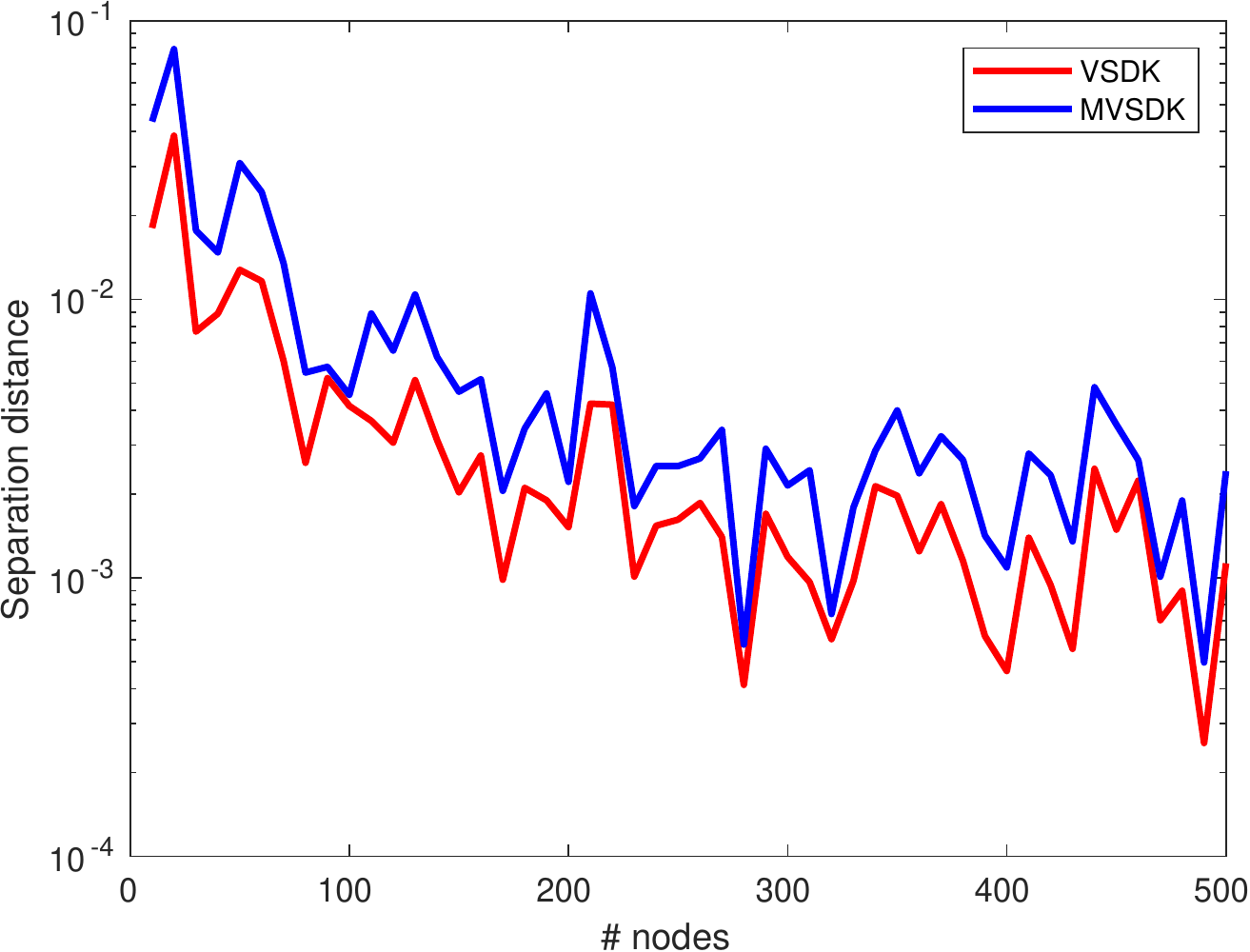}
    \includegraphics[scale=0.40]{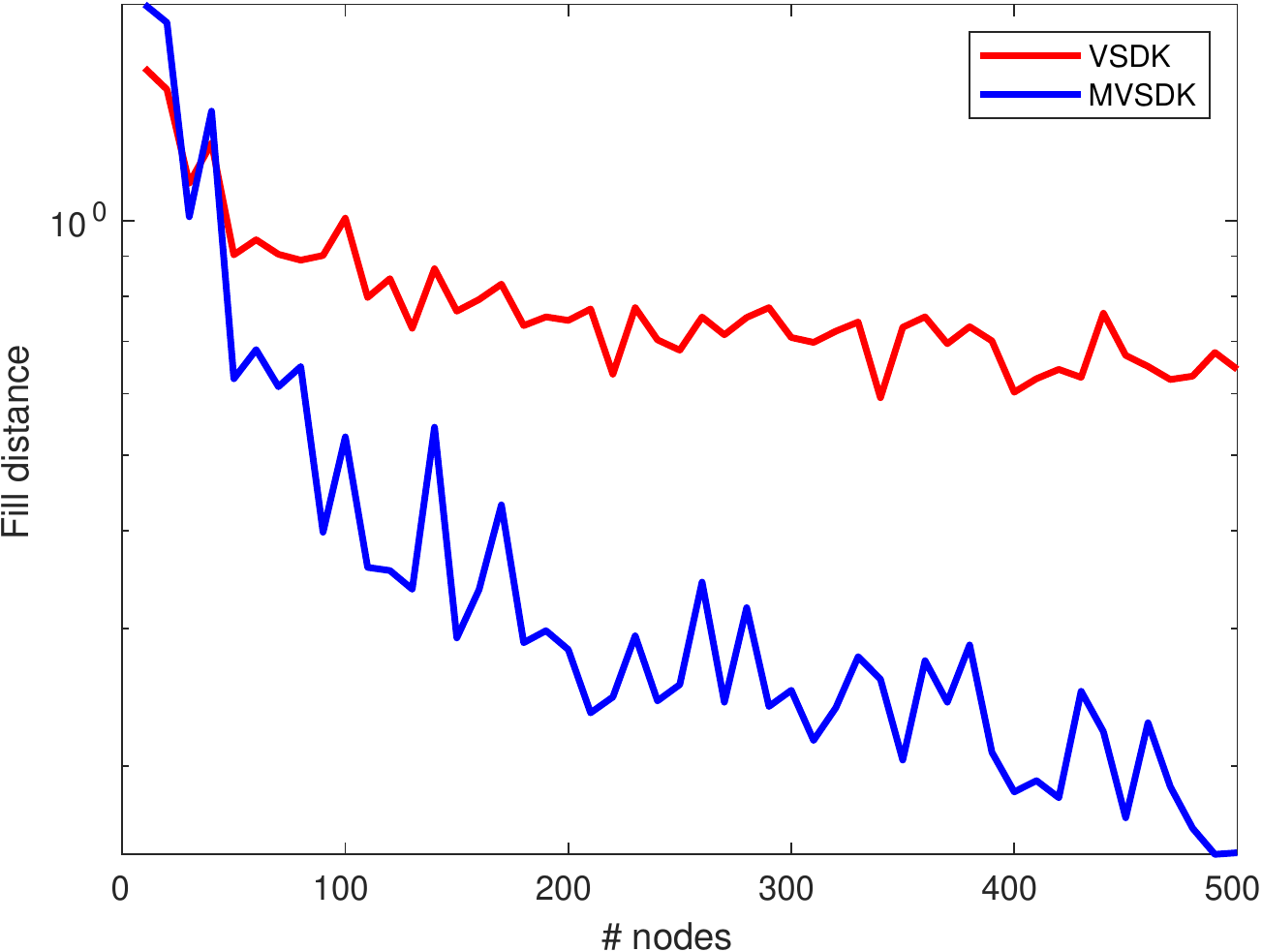}      
    \caption{The separation distance (left) and the fill distance (right) varying $N$.}
    \label{fig:sep_fill}
\end{figure}

\begin{figure}
    \centering
    \includegraphics[scale=0.40]{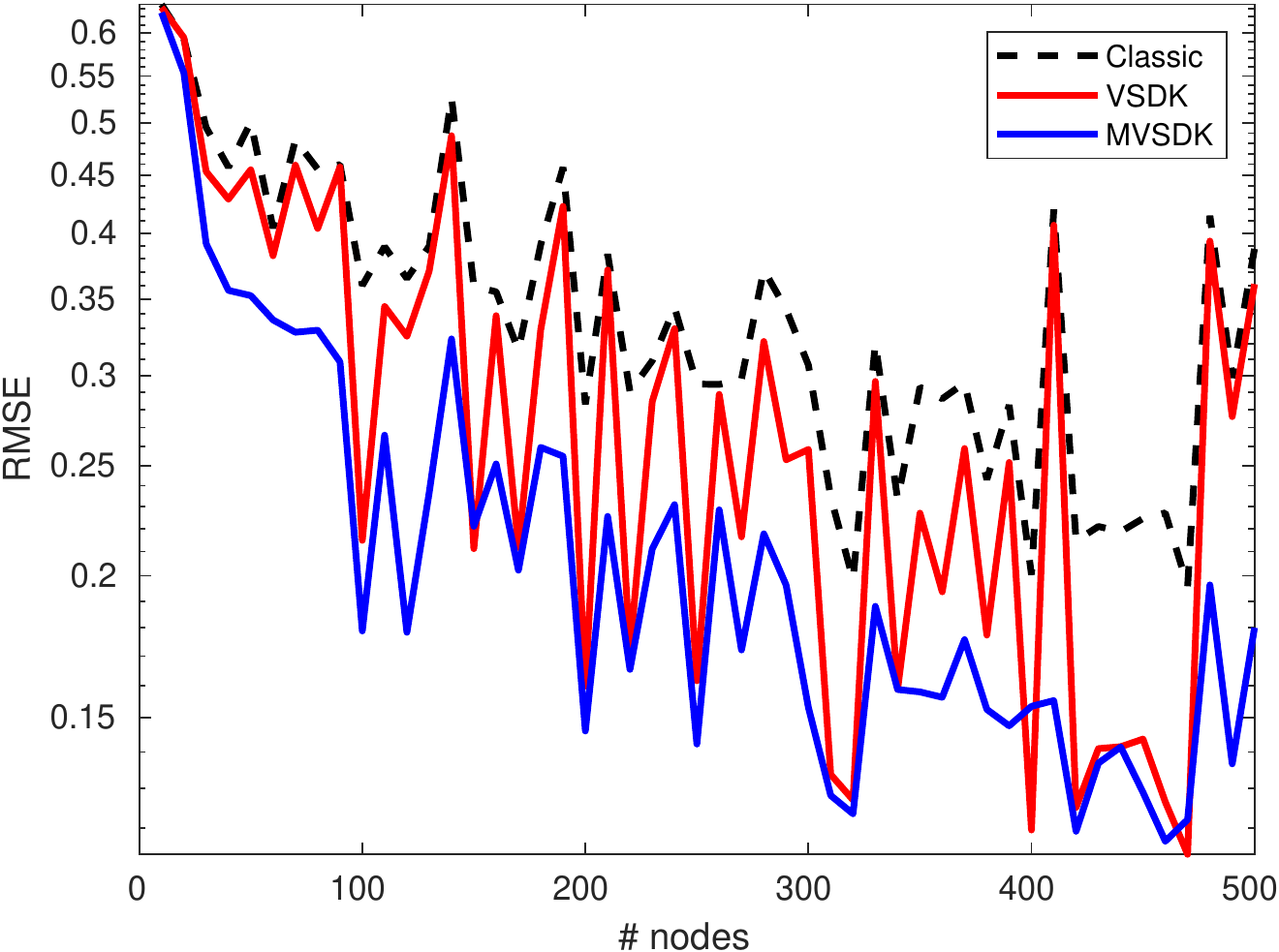}
    \includegraphics[scale=0.40]{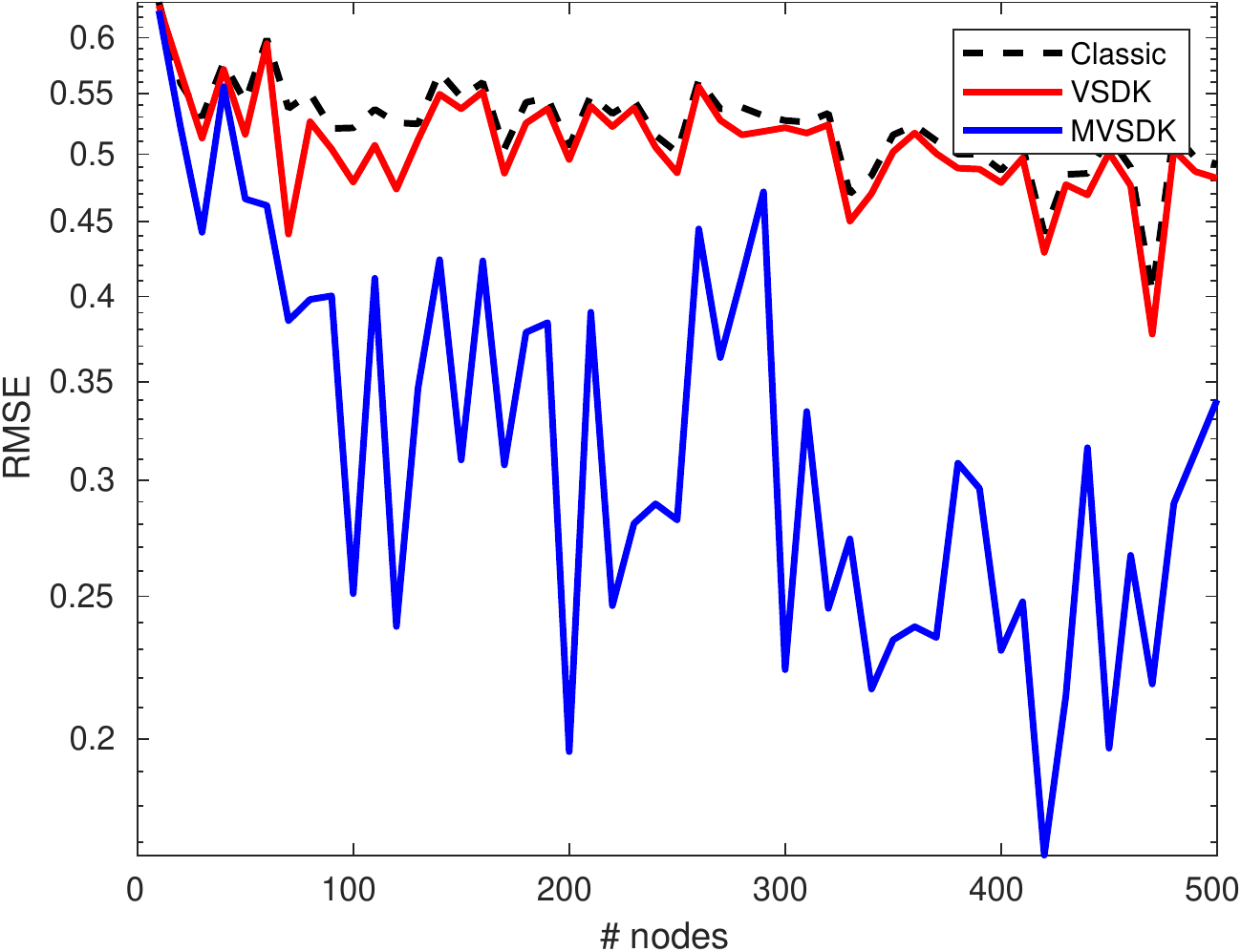}      
    \caption{The RMSE obtained using $\varphi_W$ (left) and $\varphi_M$ (right) varying $N$.}
    \label{fig:rmse}
\end{figure}

The plots in Figure \ref{fig:sep_fill} display the benefits in employing $S$ in the MVSDK in terms of diminished fill distance and increased separation distance. In Figure \ref{fig:rmse}, we observe that MVSDKs are more effective in the case of the chosen Matérn kernel. This is due to the fact that this kernel is more regular than the chosen Wendland kernel, therefore it is more prone to provide an ill-conditioned interpolation process. Furthermore, we remark that the sets of nodes $G_N$, $N=10,\dots,500$ are not nested, and they are clustered around the origin. Therefore, we can not expect an accurate recovering of the function $f$, nor convergence increasing $N$.

\section{Applications to the STIX Imaging Framework}\label{sec:stix}

In order to test the proposed MVSKs in real applied sciences, we focus on solar hard X-ray imaging and, specifically, on the ESA {{STIX}} telescope \cite{refId0}, on board of Solar Orbiter mission (see Figure \ref{fig:0}). Hard X-ray telescopes provide experimental measurements, named visibilities,  of the Fourier transform of the incoming photon flux at specific points of the spatial frequency plane.
In the case of STIX, $30$ subcollimators relying on the Moiré pattern technology provide $N=60$ visibilities on 10 circles of the frequency plane with increasing radii from about $2.79\times 10^{-3}$ arcsec$^{-1}$ to $7.02\times 10^{-2}$ arcsec$^{-1}$ (see Figure \ref{fig:0}). We observe that the visibilities lying in the lower half plane are obtained by reflecting the visibilities in the upper half with respect to the origin.

\begin{figure}
    \centering
    \raisebox{0.3\height}{\includegraphics[scale = .25]{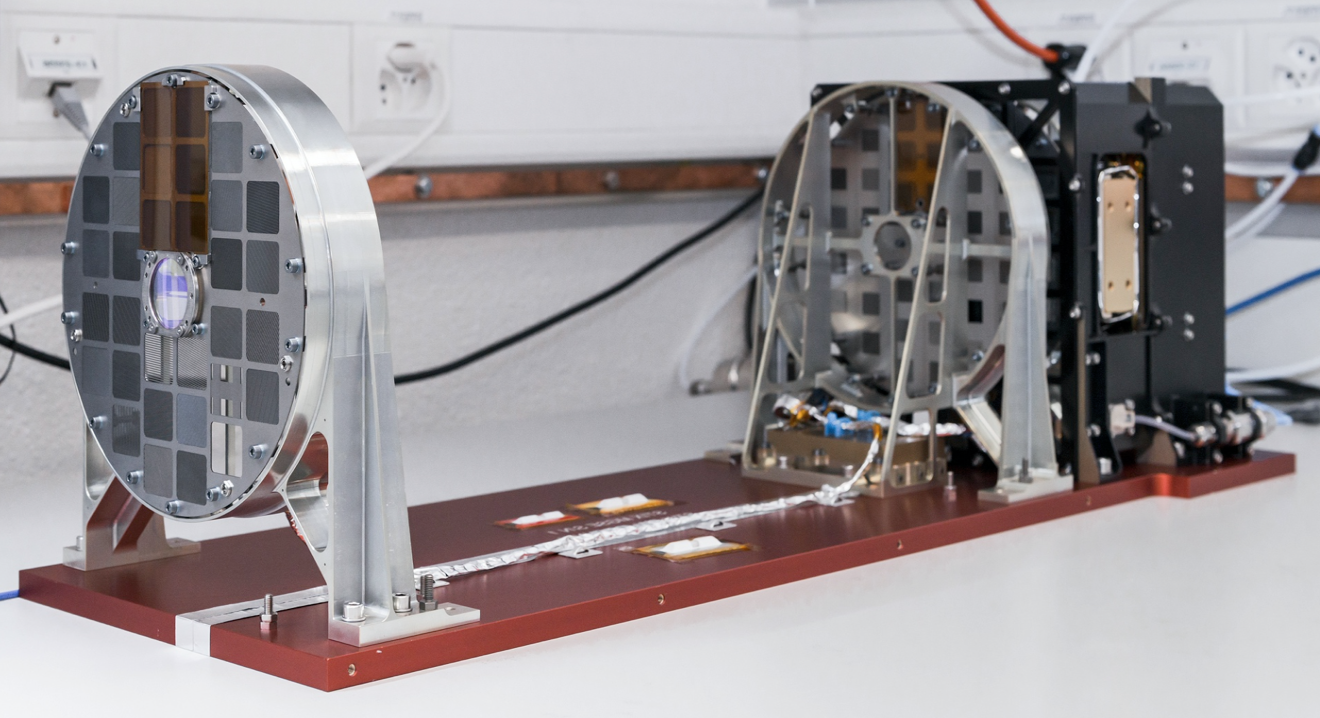}}
    \caption{The Spectrometer/Telescope for Imaging X-rays (STIX).}
    \label{fig:0}
\end{figure}

\begin{figure}
    \centering
    \includegraphics[scale=0.40]{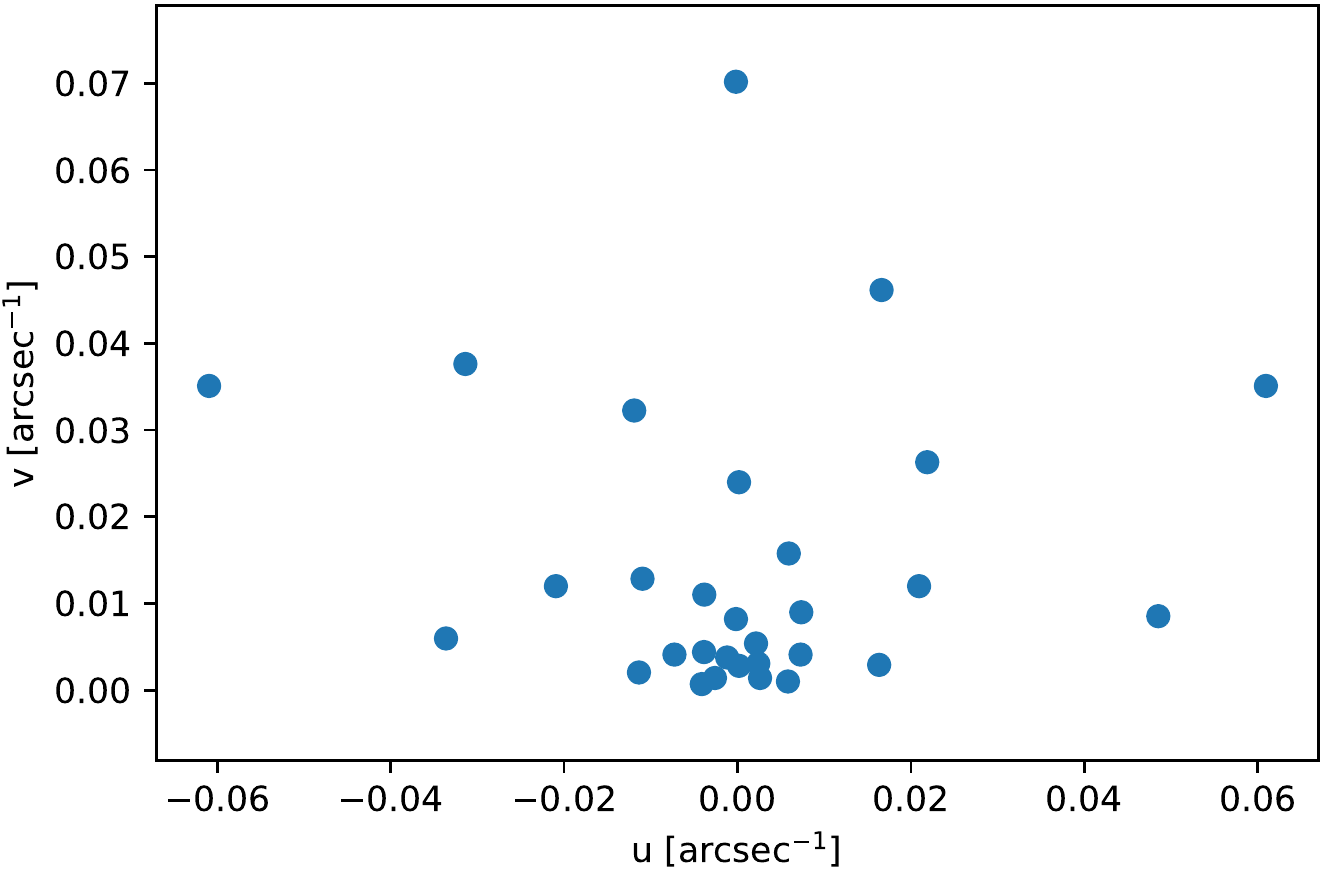}
    \includegraphics[scale=0.40]{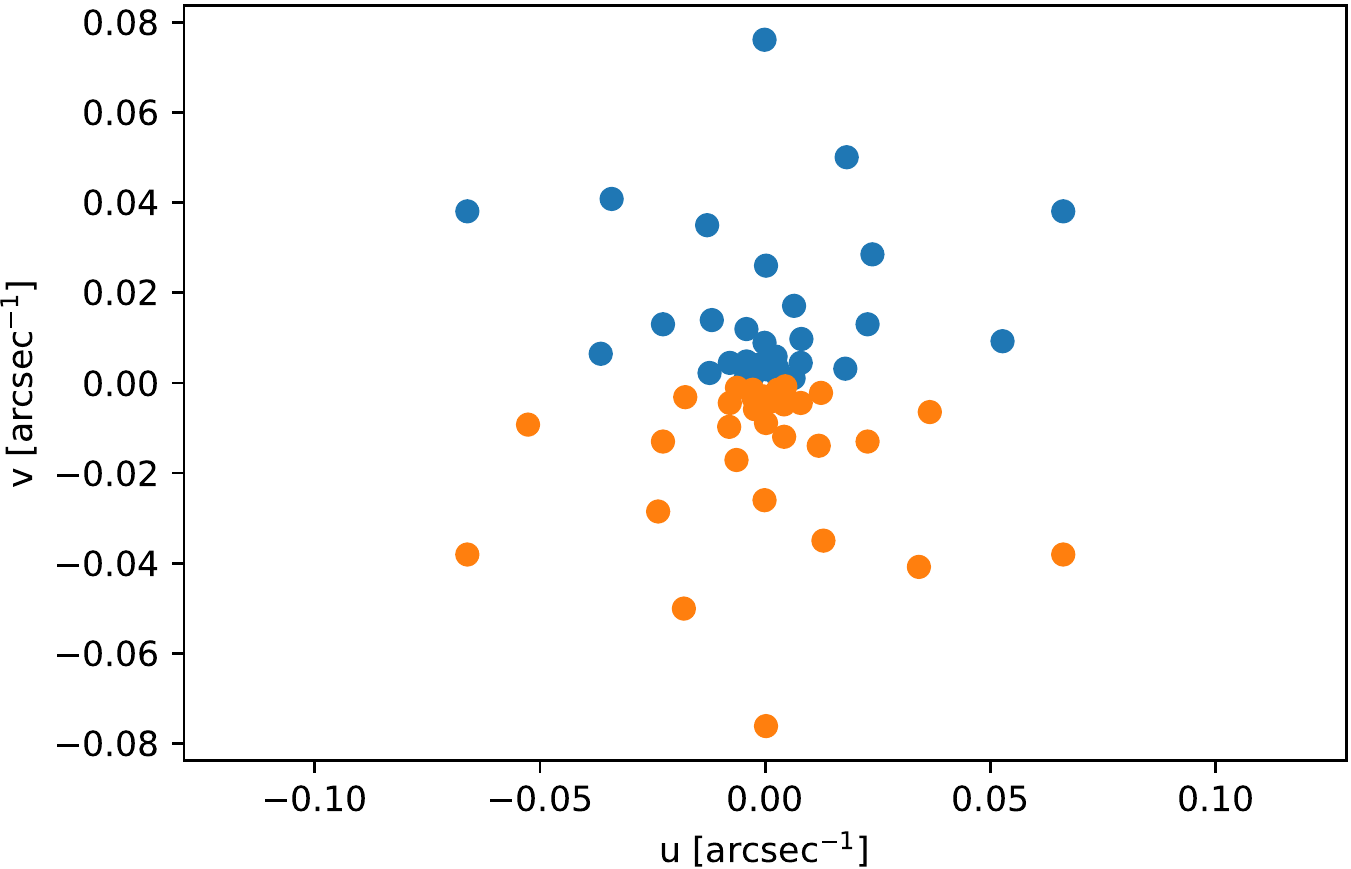} 
    \caption{STIX visibilities: reflecting with respect to the origin leads to a total number of $60$ visibilities.}
    \label{fig:0.1}
\end{figure}

We denote by ${\bf{f}}$ the vector whose components are the discretized values of the incoming flux, by ${\bf{F}}$ the discretized Fourier transform sampled at the set of  points $ \{ {\bf u}_i=(u_i,v_i) \}_{i=1}^{N}$  in the $({\mbox{u}},{\mbox{v}})$-plane and with ${\bf{V}}$ the vector whose $n$ components are the observed visibilities. Then, the image formation model in this framework can be approximated by
\begin{equation}\label{b1}
{\bf{V}} = {\bf{F}} {\bf{f}}~.
\end{equation}

\subsection{The imaging process}

Many inversion methods have been formulated to express the STIX observations as images, see e.g.  \cite{Aschwanden,Benvenuto,Bonettini,Cornwell,Massa_2020,Massone2009HARDXI,felix2017compressed}. The approach that we propose consists of two steps: interpolation of the visibilities so that we obtain the visibility surfaces and the inversion of the so generated surfaces with rather standard techniques. This idea was already used for the dismissed NASA telescope Reuven Ramaty High Energy Solar Spectroscopic Imager (RHESSI) \cite{enlighten1658}, and its IDL implementation  called uv$\_$smooth \cite{Massone2009HARDXI}, can be find in the NASA {{Solar SoftWare (SSW)}} tree. 

The uv$\_$smooth code addresses equation (\ref{b1}) by means of an interpolation and extrapolation procedure in which the interpolation step is carried out via an algorithm based on spline functions and the extrapolation step is realized by means of a soft-thresholding scheme \cite{Daubechies,Massone1}. As far as the second step is concerned, we will rely on the well-established projected Landweber iterative method \cite{Piana_1997}. The first step instead will be carried out with our MVSKs. 

\subsection{MVSKs for STIX}

To employ MVSKs, we need to define a scaling function $\psi$ and a mapping $S$. 

\subsubsection{The choice of $\psi$.} In order to define the scaling function, we take advantage of a first approximation of the inverse problem obtained via a standard back-projection algorithm \cite{Mersereau} that computes the discretized inverse Fourier transform of the visibilities by means of the IDL source code  {\texttt {vis\char`_bpmap}}  available in the NASA SSW tree. The so-constructed image is then forward Fourier transformed to obtain $\psi$. Once the interpolated visibility surface $\overline{\bf{V}}$ has been computed, the image reconstruction problem reads as follows
\begin{equation}\label{bbb5}
{\overline{\bf{V}}} = {\overline{\bf{F}}} {\overline{\bf{f}}}~,
\end{equation}
where ${\overline{\bf{F}}}$ is the $N^2 \times N^2$ discretized Fourier transform and ${\overline{\bf{f}}}$ is the $N^2 \times 1$ vector to reconstruct. In the following we will point out the advantages of interpolating the visibilities with our technique. 

\subsubsection{The choice of $S$.} To present in details the mapping $S$ chosen for the STIX imaging framework, let us first deepen the definition of the visibilities represented in Figure \ref{fig:0.1} (left). We have
\begin{equation*}
    u_i = (L_1+L_2)\frac{\cos(\alpha_i^f)}{\rho_i^f}-L_2 \frac{\cos(\alpha_i^r)}{\rho_i^r},\; v_i = (L_1+L_2)\frac{\sin(\alpha_i^f)}{\rho_i^f}-L_2 \frac{\sin(\alpha_i^r)}{\rho_i^r},
\end{equation*}
where $L_1=550$, $L_2=47$ and $\alpha_i^f,\rho_i^f,\alpha_i^r,\rho_i^r$ are discussed in \cite{STIX1}. We define the map
\begin{equation*}
    S(u,v)=C\log(\lVert (u,v) \lVert_2)(\cos(\arctan(v/u)),\sin(\arctan(v/u))),
\end{equation*}
where $C$ is a normalizing factor used to retain the order of magnitude of the original visibilities after the mapping. In Figure \ref{fig:0.2}, we can observe that the mapped visibilities are distributed in a circular crown with no clustering around the origin.

\begin{figure}
    \centering
    \includegraphics[scale=0.40]{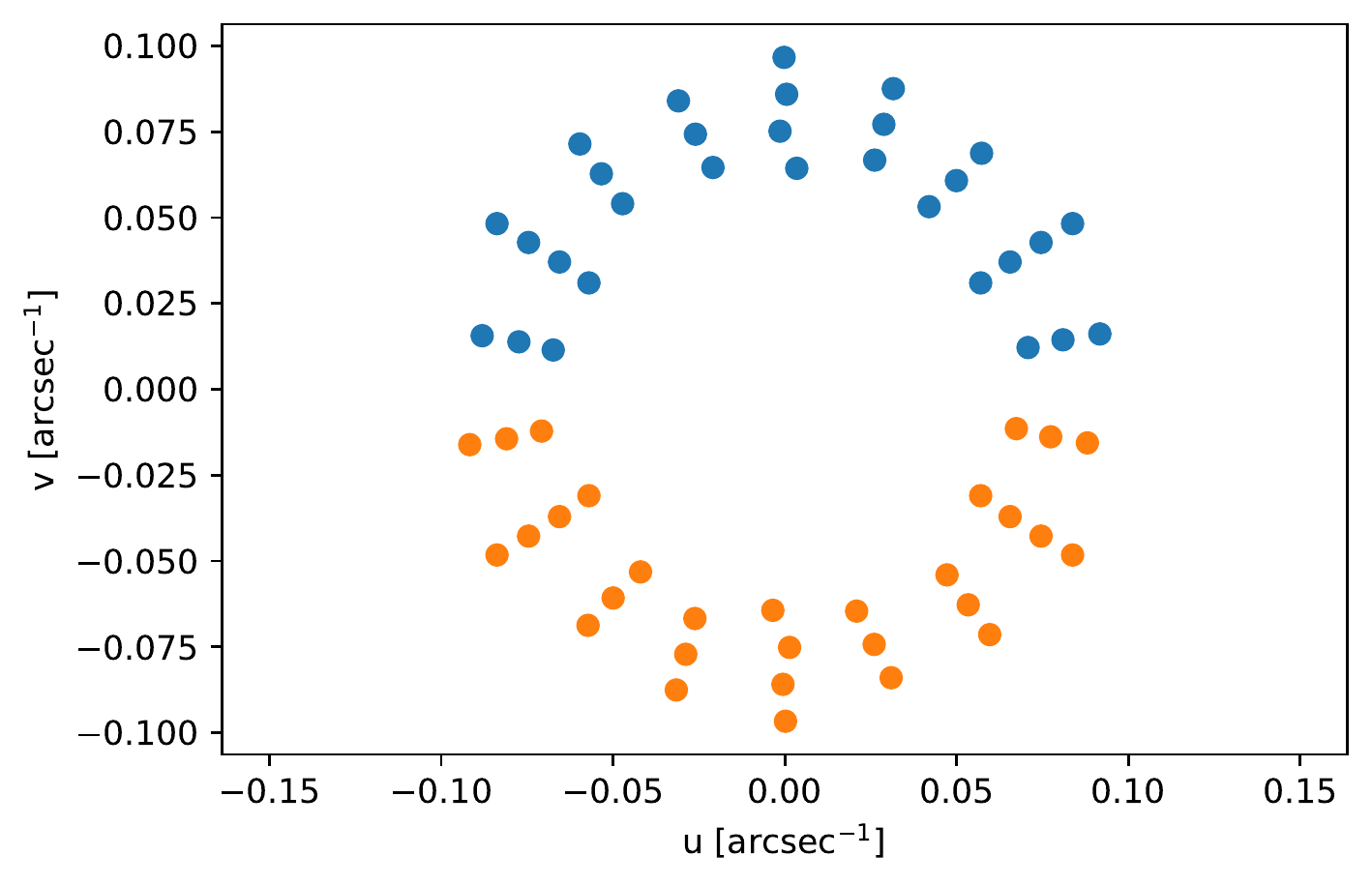}
    \caption{Mapped STIX visibilities.}
    \label{fig:0.2}
\end{figure}

\subsection{Numerical results}

On Jul 2022 STIX recorded a flare during the during the time interval 23:41:15--23:41:41 UT.  The energy range of the event is 15--25 keV. In Figure  \ref{fig:1} we reported the reconstruction carried out with the interpolation/extrapolation algorithm where we respectively interpolate with: classical radial kernels, VSKs and MVSKs. Moreover, as a further comparison we also consider mem$\_$ge \cite{Massa_2020} which is a well established implementation of the maximum entropy approach and it is used by the solar physics community.  We note that all methods present artifacts but the shape of the source reconstructed by interpolating with MVSKs is more similar to the one computed with mem$\_$ge. To have a quantitative feedback on the accuracy, we show in Figure \ref{fig:2} the visibility fits obtained with the four different approaches. We further observe that the chi squares values of mem$\_$ge and uv$\_$smooth + MVSKs are similar and are the lowest. 


\begin{figure}
    \centering
    \includegraphics[scale=0.10]{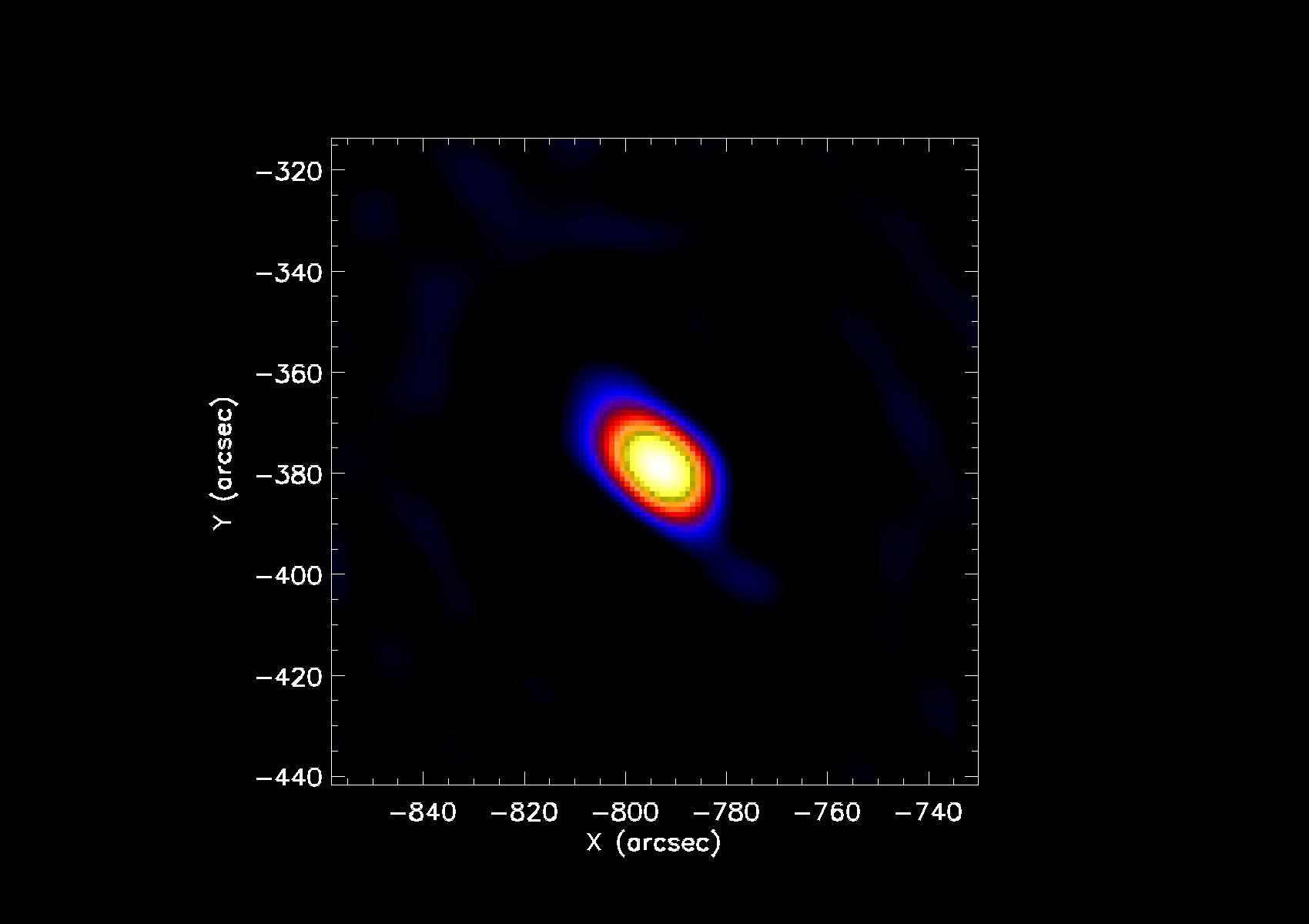}
    \includegraphics[scale=0.10]{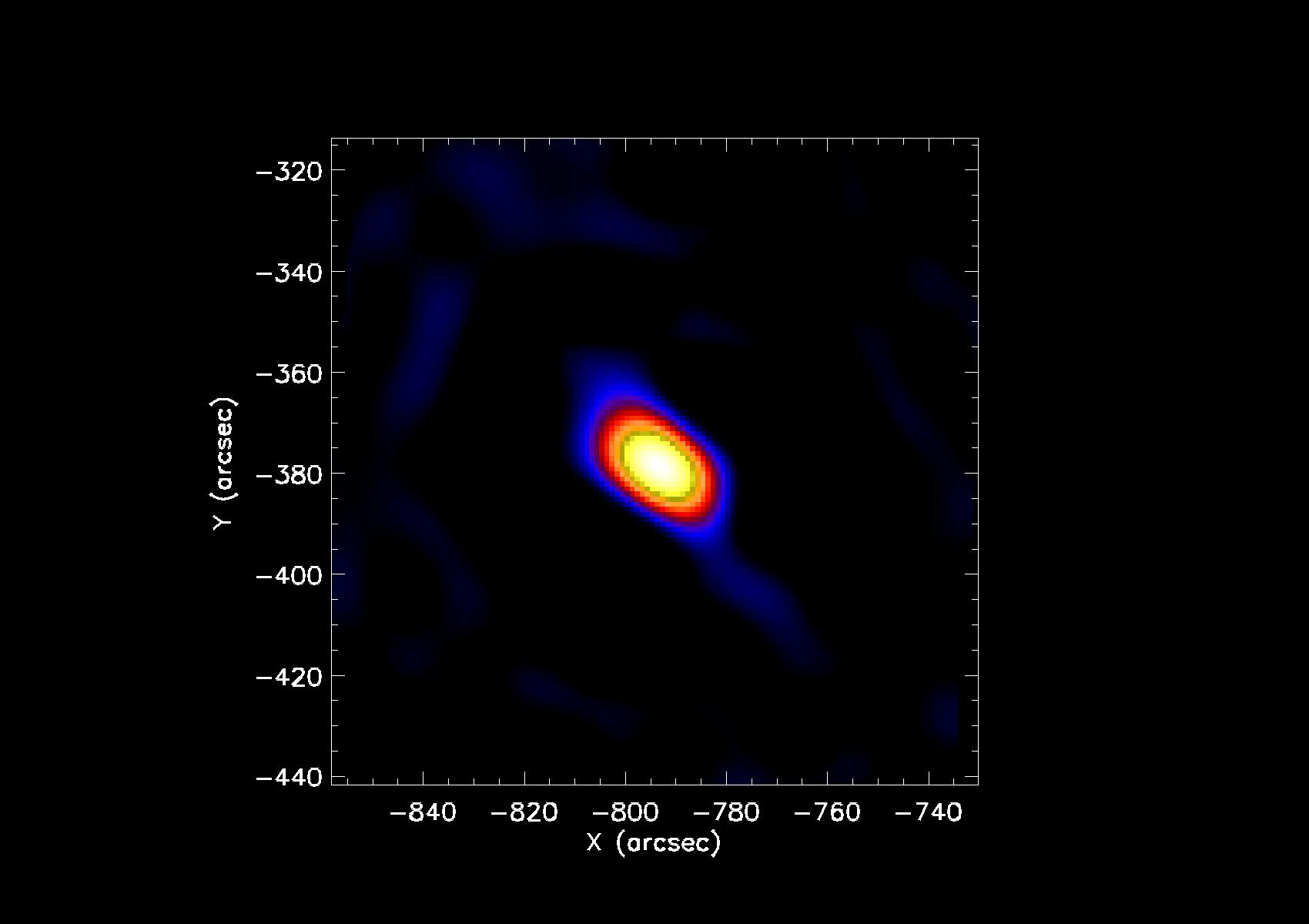} 
     \includegraphics[scale=0.10]{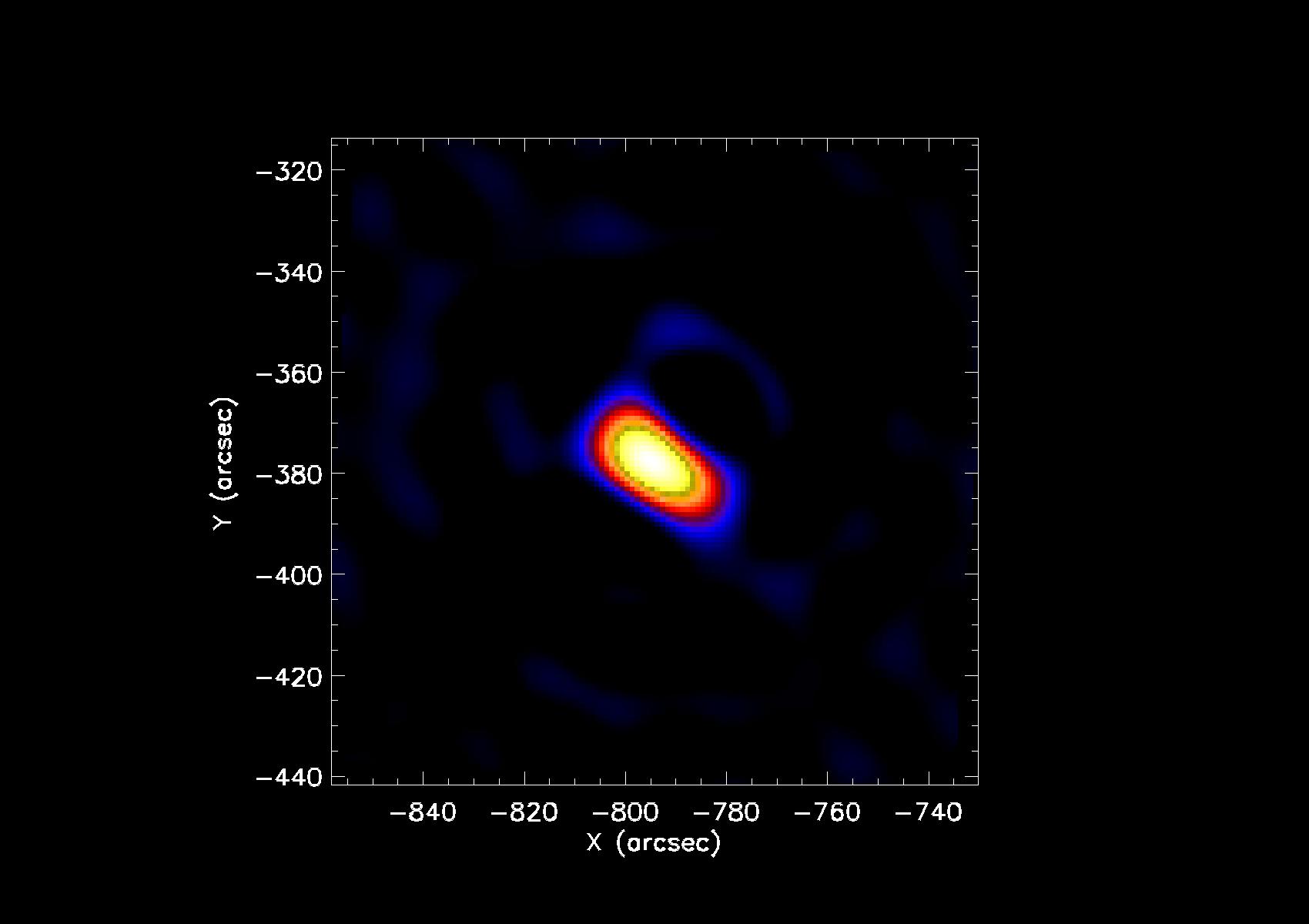}       
    \includegraphics[scale=0.10]{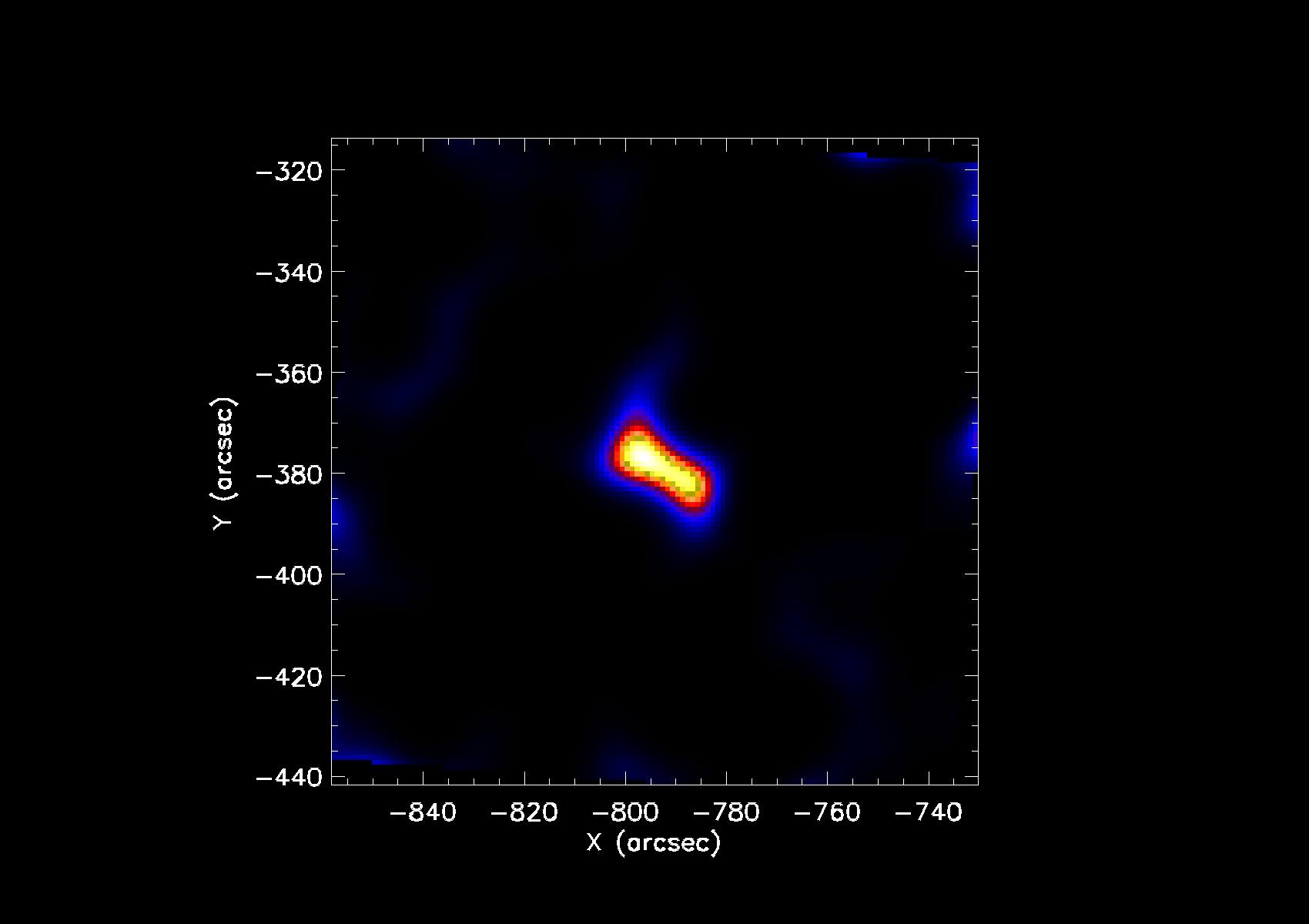}  
    \caption{Left to right, top to bottom: reconstruction of the flaring source with: uv$\_$smooth + classical radial kernels, uv$\_$smooth + VSKs, uv$\_$smooth + MVSKs and mem$\_$ge.}
    \label{fig:1}
\end{figure}

\begin{figure}
    \centering
    \includegraphics[scale=0.22]{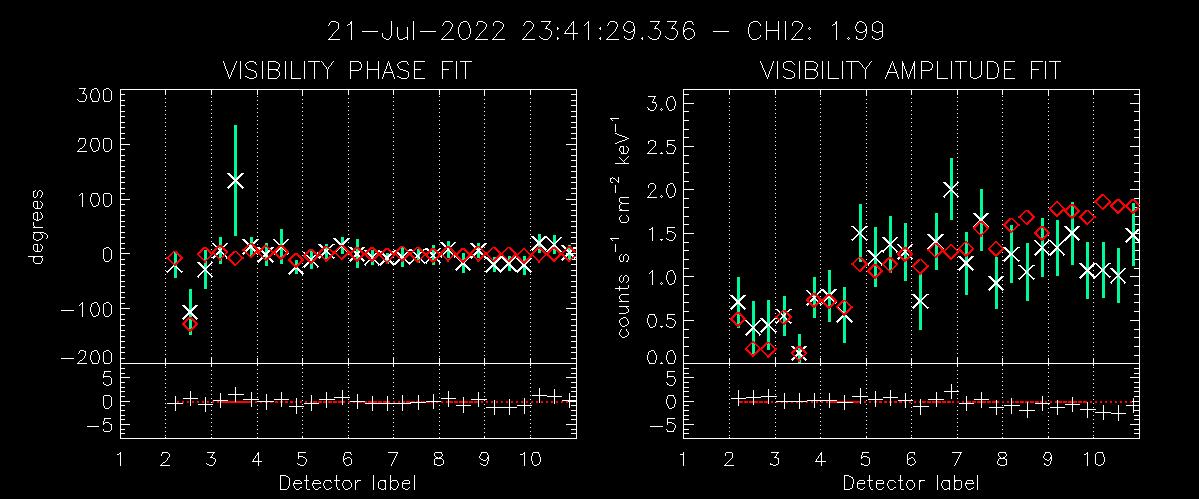}
    \includegraphics[scale=0.22]{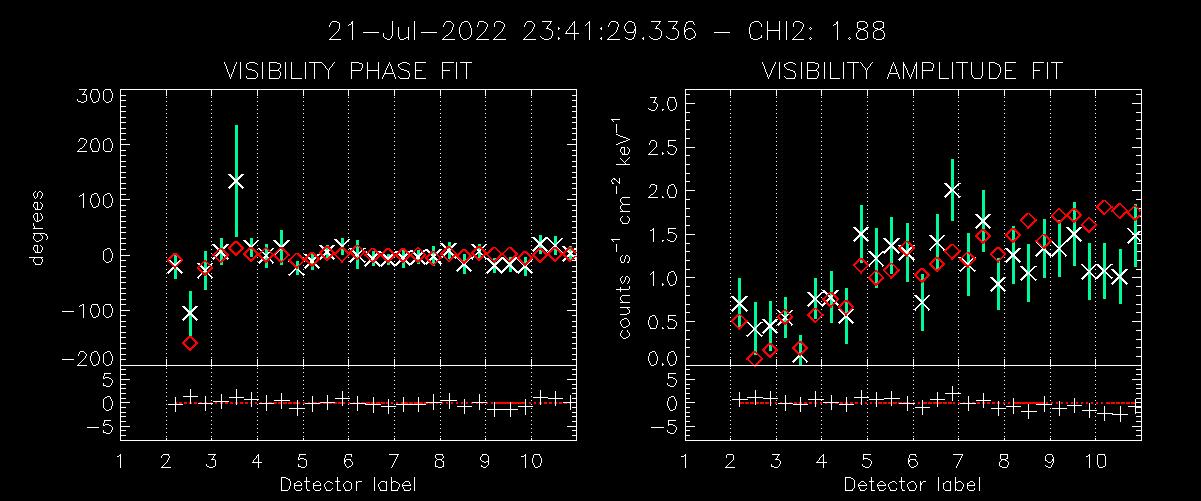} 
     \includegraphics[scale=0.22]{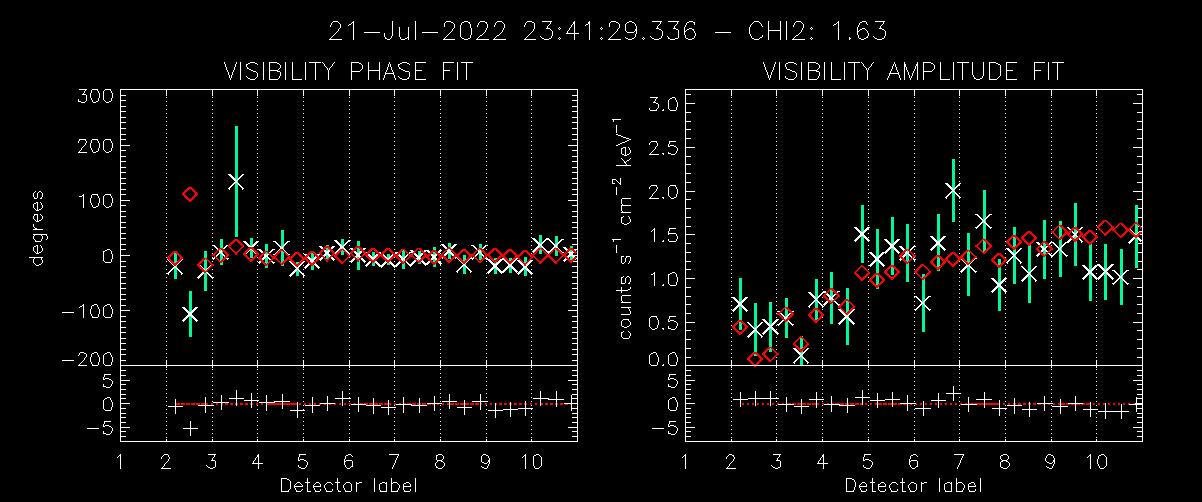}   
        \includegraphics[scale=0.22]{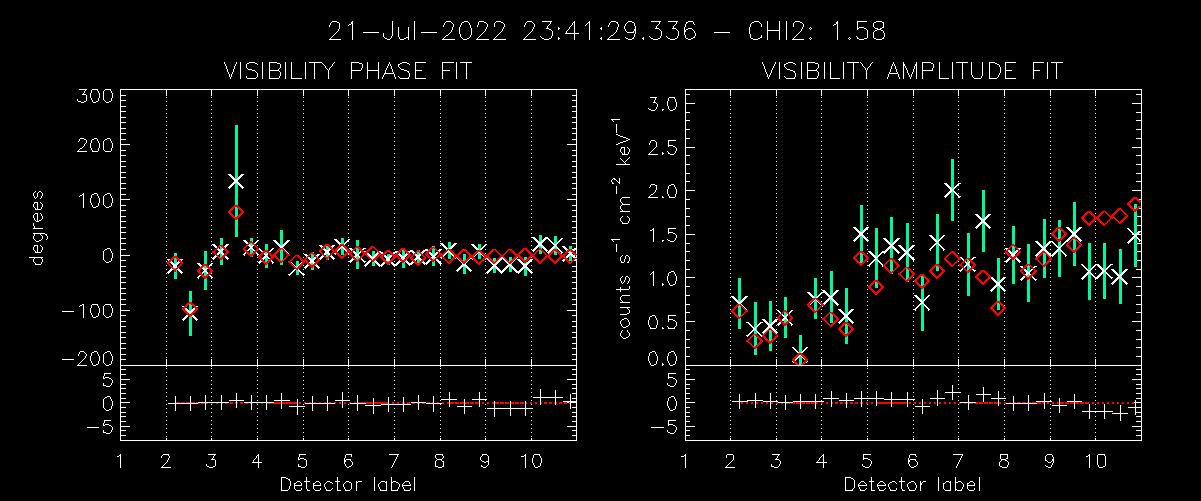}  
    \caption{Top to bottom: visibility fits of the flaring source with: uv$\_$smooth + classical radial kernels, uv$\_$smooth + VSKs, uv$\_$smooth + MVSKs and mem$\_$ge.}
    \label{fig:2}
\end{figure}



\section{Conclusions}\label{sec:conclusions}

In this work, we define MVSKs by mixing VSKs with the FNA. After providing some theoretical details, we performed some numerical simulations that showed the advantages of using MVSDKs with respect to classical RBF kernels and VSDKs. Then, we applied our method in the framework of hard X-ray astronomical imaging. The obtained results encourage further investigations on this research line.

\section*{Acknowledgments}
EP and AM are supported by the project: \lq\lq Physics-based AI for predicting extreme weather and spaceweather events (AIxtreme)\rq\rq, Founded by  Fondazione Compagnia di San Paolo. 
FM acknowledges the financial support of the Programma Operativo Nazionale (PON) "Ricerca e Innovazione" 2014 - 2020. The authors kindly acknowledge the financial contribution from the AI-FLARES ASI-INAF n.2018-16-HH.0 agreement. This research has been accomplished within GNCS-IN$\delta$AM.

%
%
%
 \bibliographystyle{splncs04}
 \bibliography{bibstix}

\begin{thebibliography}{10}
\providecommand{\url}[1]{\texttt{#1}}
\providecommand{\urlprefix}{URL }
\providecommand{\doi}[1]{https://doi.org/#1}

\bibitem{Massone1}
Allavena, S., Piana, M., Benvenuto, F., Massone, A.M.: An
  interpolation/extrapolation approach to {X}-ray imaging of solar flares.
  Inverse Probl. Imag.  \textbf{6}, ~147 (2012)

\bibitem{Aschwanden}
Aschwanden, M.J., Schmal, E., the {RHESSI}~Team: Reconstruction of {RHESSI}
  solar flare images with a forward-fitting method. Sol. Phys.  \textbf{210},
  193 -- 211 (2002)

\bibitem{Benvenuto}
Benvenuto, F., Schwartz, R., Piana, M., Massone, A.M.: {Expectation
  maximization for hard {X}-ray count modulation profiles}. Astron. Astrophys.
  \textbf{555}, ~A61 (2013)

\bibitem{Berrut20}
Berrut, J.P., De~Marchi, S., Elefante, G., Marchetti, F.: Treating the {G}ibbs
  phenomenon in barycentric rational interpolation and approximation via the
  {$S$}-{G}ibbs algorithm. Appl. Math. Lett.  \textbf{103},  106196, 7 (2020)

\bibitem{Bonettini}
Bonettini, S., Anastasia, C., Prato, M.: {A new semiblind deconvolution
  approach for Fourier-based image restoration: An application in astronomy}.
  SIAM J. Imaging Sci.  \textbf{6},  1736--1757 (2013)

\bibitem{Bozzini15}
Bozzini, M., Lenarduzzi, L., Rossini, M., Schaback, R.: Interpolation with
  variably scaled kernels. IMA J. Numer. Anal.  \textbf{35}(1),  199--219
  (2015)

\bibitem{BLSE:RN017988920}
Brutman, L.: Lebesgue functions for polynomial interpolation - a survey. Ann.
  Numer. Math.  \textbf{4}(1/4),  111--128 (1996)

\bibitem{campivsk}
Campi, C., Marchetti, F., Perracchione, E.: Learning via variably scaled
  kernels. Adv. Comput. Math.  \textbf{47}(4),  Paper No. 51, 23 (2021).
  \doi{10.1007/s10444-021-09875-6},
  \url{https://doi.org/10.1007/s10444-021-09875-6}

\bibitem{Cornwell}
Cornwell, T., Evans, K.F.: A simple maximum entropy deconvolution algorithm.
  Astron. Astrophys.  \textbf{143}(1),  77--83 (1985)

\bibitem{Daubechies}
Daubechies, I., Defrise, M., De~Mol, C.: An iterative thresholding algorithm
  for linear inverse problems with a sparsity constraint. Commun. Pure Appl.
  Math.  \textbf{57}(11),  1413--1457 (2004)

\bibitem{DeMarchi21b}
De~Marchi, S., Elefante, G., Marchetti, F.: Stable discontinuous mapped bases:
  the {G}ibbs-{R}unge-avoiding stable polynomial approximation ({GRASPA})
  method. Comput. Appl. Math.  \textbf{40}(8),  Paper No. 299, 17 (2021)

\bibitem{vskmpi}
De~Marchi, S., Erb, W., Marchetti, F., Perracchione, E., Rossini, M.:
  Shape-driven interpolation with discontinuous kernels: Error analysis, edge
  extraction, and applications in magnetic particle imaging. SIAM J. Sci.
  Comput.  \textbf{42}(2),  B472--B491 (2020)

\bibitem{vskjump}
De~Marchi, S., Marchetti, F., Perracchione, E.: Jumping with {V}ariably
  {S}caled {D}iscontinuous {K}ernels ({VSDK}s). BIT Numer. Math.  \textbf{60},
  441--463 (2020)

\bibitem{DeMarchi20}
De~Marchi, S., Marchetti, F., Perracchione, E., Poggiali, D.: Polynomial
  interpolation via mapped bases without resampling. J. Comput. Appl. Math.
  \textbf{364},  112347, 12 (2020)

\bibitem{DeMarchi21}
De~Marchi, S., Marchetti, F., Perracchione, E., Poggiali, D.: Multivariate
  approximation at fake nodes. Appl. Math. Comput.  \textbf{391},  Paper No.
  125628, 17 (2021)

\bibitem{DeMarchi22}
{De Marchi}, S., Lot, F., Marchetti, F., Poggiali, D.: Variably scaled
  persistence kernels ({VSPK}s) for persistent homology applications. Journal
  of Computational Mathematics and Data Science  \textbf{4},  100050 (2022)

\bibitem{Fasshauer}
Fasshauer, G.E.: Meshfree Approximation Methods with MATLAB. World Scientific
  (2007)

\bibitem{felix2017compressed}
Felix, S., Bolzern, R., Battaglia, M.: A compressed sensing-based image
  reconstruction algorithm for solar flare x-ray observations. Astrophys. J.
  \textbf{849}(1), ~10 (2017)

\bibitem{STIX1}
Giordano, S., Pinamonti, N., Piana, M., Massone, A.M.: The process of data
  formation for the {S}pectrometer/{T}elescope for {I}maging {X}-rays {(STIX)}
  in {S}olar {O}rbiter. SIAM J. Imaging Sci.  \textbf{8}(2),  1315--1331 (2015)

\bibitem{Guastavino20}
Guastavino, S., Benvenuto, F.: Convergence rates of spectral regularization
  methods: a comparison between ill-posed inverse problems and statistical
  kernel learning. SIAM J. Numer. Anal.  \textbf{58}(6),  3504--3529 (2020)

\bibitem{Esfahani23}
Karimnejad~Esfahani, M., De~Marchı, S., Marchetti, F.: Moving least squares
  approximation using variably scaled discontinuous weight function.
  Constructive Mathematical Analysis  \textbf{6}(1),  38 -- 54 (2023)

\bibitem{enlighten1658}
Lin, R.P., Dennis, B.R., Hurford, G.J., Smith, D.M., Zehnder, A., Harvey, P.R.,
  Curtis, D.W., Pankow, D., Turin, P., Bester, M., Csillaghy, A., Lewis, M.,
  Madden, N., Beek, H.F.V., Appleby, M., Raudorf, T., McTiernan, J., Ramaty,
  R., Schmahl, E., Schwartz, R., Krucker, S., Abiad, R., Quinn, T., Berg, P.,
  Hashii, M., Sterling, R., Jackson, R., Pratt, R., Campbell, R.D., Malone, D.,
  Landis, D., Barrington-Leigh, C.P., Slassi-Sennou, S., Cork, C., Clark, D.,
  Amato, D., Orwig, L., Boyle, R., Banks, I.S., Shirey, K., Tolbert, A.K.,
  Zarro, D., Snow, F., Thomsen, K., Henneck, R., Mchedlishvili, A., Ming, P.,
  Fivian, M., Jordan, J., Wanner, R., Crubb, J., Preble, J., Matranga, M.,
  Benz, A., Hudson, H., Canfield, R.C., Holman, G.D., Crannell, C., Kosugi, T.,
  Emslie, A.G., Vilmer, N., Brown, J.C., Johns-Krull, C., Aschwanden, M.,
  Metcalf, T., Conway, A.: The {R}euven {R}amaty {H}igh-{E}nergy {S}olar
  {S}pectroscopic {I}mager ({RHESSI}). Sol. Phy.  \textbf{210}(1-2),  3 --32
  (2002)

\bibitem{Ling22}
Ling, L., Marchetti, F.: A stochastic extended {R}ippa’s algorithm for
  {L}p{OCV}. Applied Mathematics Letters  \textbf{129},  107955 (2022)

\bibitem{Massa_2020}
Massa, P., Schwartz, R., Tolbert, A.K., Massone, A.M., Dennis, B.R., Piana, M.,
  Benvenuto, F.: {MEM}{\_}{GE}: {A} new maximum entropy method for image
  reconstruction from solar {X}-ray visibilities. Astrophys. J.
  \textbf{894}(1), ~46 (2020)

\bibitem{Massone2009HARDXI}
Massone, A.M., Emslie, A.G., Hurford, G.J., Prato, M., Kontar, E.P., Piana, M.:
  Hard {X}-ray imaging of solar flares using interpolated visibilities.
  Astrophys. J.  \textbf{703},  2004--2016 (2009)

\bibitem{Mersereau}
Mersereau, R., Oppenheim, A.: {Digital reconstruction of multidimensional
  signals from their projections}. Proc. of IEEE  \textbf{62}(10) (1974)

\bibitem{Piana_1997}
Piana, M., Bertero, M.: Projected {L}andweber method and preconditioning.
  Inverse Probl.  \textbf{13}(2),  441--463 (1997)

\bibitem{refId0}
{S. Krucker}, {G. J. Hurford}, {O. Grimm}, {S. K\"ogl}, {H. P. Gr\"obelbauer},
  {L. Etesi}, {D. Casadei}, {A. Csillaghy}, {A. O. Benz}, {N. G. Arnold}, {F.
  Molendini}, {P. Orleanski}, {D. Schori}, {H. Xiao}, {M. Kuhar}, {N.
  Hochmuth}, {S. Felix}, {F. Schramka}, {S. Marcin}, {S. Kobler}, {L. Iseli},
  {M. Dreier}, {H. J. Wiehl}, {L. Kleint}, {M. Battaglia}, {E. Lastufka}, {H.
  Sathiapal}, {K. Lapadula}, {M. Bednarzik}, {G. Birrer}, {St. Stutz}, {Ch.
  Wild}, {F. Marone}, {K. R. Skup}, {A. Cichocki}, {K. Ber}, {K. Rutkowski},
  {W. Bujwan}, {G. Juchnikowski}, {M. Winkler}, {Darmetko}, {M. Michalska}, {K.
  Seweryn}, {A. Bialek}, {P. Osica}, {J. Sylwester}, {M. Kowalinski}, {D.
  \'{}Scislowski}, {M. Siarkowski}, {M. Ste\'{}slicki}, {T. Mrozek}, {P.
  Podg\'orski}, {A. Meuris}, {O. Limousin}, {O. Gevin}, {I. Le Mer}, {S. Brun},
  {A. Strugarek}, {N. Vilmer}, {S. Musset}, {M. Maksimovi\'{}c}, {F.
  F\'arn\'{\i}k}, {Z. Koz\'acek}, {J. Kasparov\'a}, {G. Mann}, {H. \"Onel}, {A.
  Warmuth}, {J. Rendtel}, {J. Anderson}, {S. Bauer}, {F. Dionies}, {J.
  Paschke}, {D. Pl\"uschke}, {M. Woche}, {F. Schuller}, {A. M Veronig}, {E. C.
  M. Dickson}, {P. T. Gallagher}, {S. A. Maloney}, {D. S. Bloomfield}, {M.
  Piana}, {A. M. Massone}, {F. Benvenuto}, {P. Massa}, {R. A. Schwartz}, {B. R.
  Dennis}, {H. F. van Beek}, {J. Rodr\'{\i}guez-Pacheco}, {R. P. Lin}: The
  {S}pectrometer/{T}elescope for {I}maging {X-}rays ({STIX}). Astron.
  Astrophys.  \textbf{642}, ~A15 (2020)

\bibitem{Saitoh16}
Saitoh, S., Sawano, Y.: Theory of reproducing kernels and applications,
  Developments in Mathematics, vol.~44. Springer, Singapore (2016)

\bibitem{Wendland05}
Wendland, H.: Scattered {D}ata {A}pproximation, Cambridge Monographs on Applied
  and Computational Mathematics, vol.~17. Cambridge University Press, Cambridge
  (2005)

\end{thebibliography}

\end{document}